%% file: main.tex
\documentclass[sort&compress]{elsarticle}

\usepackage{amsmath,amsfonts,amsthm}
\usepackage{multirow}
\usepackage{pgfplots}
\pgfplotsset{compat=newest}
\usetikzlibrary{plotmarks}
\usetikzlibrary{arrows.meta}
\usepgfplotslibrary{patchplots}
\usepackage{grffile}
\usepackage[linesnumbered,ruled]{algorithm2e}
\newtheorem{theorem}{Theorem}
\newtheorem{remark}{Remark}
\newcommand{\rmi}{\mathrm{i}}
\newcommand{\rme}{\mathrm{e}}

\newcommand{\CC}{\mathbb{C}}
\newcommand{\bb}{\boldsymbol}
\usepackage{hyperref}

\begin{document}

\begin{frontmatter}
\title{Efficient third order tensor-oriented directional splitting for exponential integrators}
\author[1]{Fabio Cassini\corref{cor1}}
\ead{fabio.cassini@univr.it}
\cortext[cor1]{Corresponding author}

\affiliation[1]{organization={Department of Computer Science,
                               University of Verona},
            addressline={Strada le Grazie, 15}, 
            city={Verona},
            postcode={37134},
            country={Italy}}
\myfooter[L]{} 

\begin{abstract}
Suitable discretizations through tensor product formulas of popular 
multidimensional operators (diffusion or diffusion--advection, for instance)
lead to matrices with $d$-dimensional Kronecker sum structure.
For evolutionary Partial Differential Equations containing such operators and integrated in time with 
exponential integrators, it is then of paramount importance to efficiently
approximate the actions of $\varphi$-functions of the arising matrices.
In this work, we show how to produce directional split approximations of
third order with respect to the time step size.
They conveniently employ tensor-matrix products (the so-called $\mu$-mode
product and related Tucker operator, realized in practice
with high performance level 3 BLAS), and allow for the effective 
usage of exponential Runge--Kutta integrators up to order three. The technique can also
be efficiently implemented on modern computer hardware such as Graphic
Processing Units.
The approach has been successfully tested against state-of-the-art techniques
on two well-known physical models that lead to Turing patterns,
namely the 2D Schnakenberg and the 3D FitzHugh--Nagumo systems, on different
{\color{black}hardware and software} architectures.
\end{abstract}


\end{frontmatter}

  \section{Introduction}
  We are interested in the solution of stiff systems of Ordinary
  Differential Equations (ODEs) of type
  \begin{subequations}\label{eq:ODE}
  \begin{equation}\label{eq:ODEeq}
      \bb u'(t)=K\bb u(t)+\bb g(t,\bb u(t))=\bb f(t,\bb u(t)),\quad
      \bb u(0)=\bb u_0,
  \end{equation}
  using exponential integrators~\cite{HO10}.
  The stiff part is represented by the matrix
  $K\in\CC^{N\times N}$ which we assume to have $d$-dimensional
  \emph{Kronecker sum structure}
  \begin{equation}\label{eq:kronsum}
    K = A_d \oplus \cdots \oplus A_1 =
    A_{\otimes 1} + \cdots +A_{\otimes d},
  \end{equation}
  with
  \begin{equation}
    A_{\otimes\mu} = I_d \otimes \cdots \otimes I_{\mu+1} \otimes A_\mu
    \otimes I_{\mu-1} \otimes \cdots \otimes I_1, \quad \mu=1,\ldots,d.
  \end{equation}
  \end{subequations}
Here, $A_\mu\in \CC^{n_\mu\times n_\mu}$  and $I_\mu$ is
  the identity matrix of size $n_\mu$. Moreover,
  $\bb g(t,\bb u(t))$ is a nonlinear function
  of $t$   and of the unknown $\bb u(t) \in \CC^N$, with $N=n_1\cdots n_d$.
  Throughout the paper the
  symbol~$\otimes$ denotes the Kronecker product of matrices,
  while~$\oplus$ is employed for the Kronecker sum of matrices.
  Systems in form~\eqref{eq:ODE} naturally arise when discretizing in space some
  Partial Differential Equations (PDEs) defined on tensor product domains
  and with appropriate boundary conditions. In those cases,
  $A_\mu$ are matrices which correspond to differential or fractional
  one-dimensional operators after discretization with (nonuniform) finite
  differences~\cite{CCEOZ22,CC24}, tensor product finite or spectral
  elements~\cite{MMPC22,CMM23}, and
  usually encapsulate boundary conditions. Typical examples are (systems of)
  evolutionary PDEs which contain diffusion--advection--absorption or
  Schr\"odinger operators, among the others.

  In the last years, exponential integrators proved to be a valuable alternative
  for the efficient integration of systems with Kronecker sum 
  structure~\eqref{eq:ODE}~\cite{CCEOZ22,CC24,MMPC22,CMM23,CCZ23,CCZ23phi,CC23,AAKW20,DASS20,LZL21,LZZ22,NWZL08,JZZD15,ZJZ16,HJW19}.
  In particular, it is shown in References~\cite{CC23,CC24} that
  the well-known exponential Runge--Kutta method of order two \textsc{etd2rk}
  \begin{equation}\label{eq:ETD2RK}
    \begin{aligned}
      \bb u_{n2}&=\bb u_n+\tau\varphi_1(\tau K)\bb f(t_n,\bb u_n),\\
      \bb u_{n+1}&=\bb u_{n2}+
      \tau\varphi_2(\tau K)(\bb g(t_{n+1},\bb u_{n2})-\bb g(t_n,\bb u_n)),
    \end{aligned}
  \end{equation}
  where $\tau$ denotes the time step size, can be directionally split, obtaining
  the so-called \textsc{etd2rkds} integrator. There, it is also presented that
  {\color{black}the latter} strongly outperforms other well-established methods in the solution of stiff
  diffusion--advection--reaction models.
  The $\varphi$-functions appearing in exponential integrators are exponential-like matrix 
  functions and, for generic matrix $X\in\CC^{N\times N}$, have Taylor expansion
  \begin{equation*}
    \varphi_\ell(X)=\sum_{k=0}^\infty\frac{X^k}{(k+\ell)!}, \quad \ell\geq 0.
  \end{equation*}
  Concerning their computation and action to a vector, we refer to
  References~\cite{AMH10,SID19,CZ19,AIDAJ23,SW09,LYL22} for algorithms suitable for small sized matrices,
  to References~\cite{GRT18,LPR19,AMH11,NW12,CKOR16,CCZ20,CCZ23b} for large and sparse matrices,
  and to References~\cite{BS17,CCEOZ22,CCZ23phi,CCZ23,CMM23,LZL21,LZZ22,MMPC22} when
  $X$ is a Kronecker sum.
  The key points to develop the directional split integrator
  based on method~\eqref{eq:ETD2RK} were the observation that
  \begin{equation}\label{eq:phisplit}
  \begin{aligned}
    \varphi_1(\tau K)&= \varphi_1(\tau A_d)\otimes\cdots\otimes\varphi_1(\tau A_1)
    +\mathcal{O}(\tau^2), \\
    \varphi_2(\tau K)&= 2^{d-1}\varphi_2(\tau A_d)\otimes\cdots\otimes\varphi_2(\tau A_1)
        +\mathcal{O}(\tau^2),
  \end{aligned}
  \end{equation}
  and that the actions of the right hand sides can be efficiently realized in
  tensor form with Tucker operators (see {\color{black}the next section and}
  References~\cite{CCEOZ22,CCZ23} for insights on the relevant tensor-matrix
  operations).
  However, the second order accuracy  of the formulas with respect to the time
  step size essentially
  limits their applicability to schemes of at most order two.

  In this paper we aim at introducing directional split approximations of
  third order with respect to $\tau$ for
  $\varphi$-functions of matrices with $d$-dimensional Kronecker sum structure.
  The formulas derived in Section~\ref{sec:approxds} allow for the efficient
  construction
  and employment, for instance, of exponential Runge--Kutta integrators of order three for ODEs systems
  with
  Kronecker sum structure. In particular, since the realization of the approximations
  heavily exploits BLAS operations, the resulting schemes can also be efficiently 
  implemented on modern hardware architectures such as multi-core Central 
  Processing Units (CPUs) and Graphic Processing Units (GPUs). 
  The effectiveness of the approach and the performance results on 
  two popular systems of diffusion--reaction equations are shown in the numerical
  examples of Section~\ref{sec:numexp}. We finally draw the conclusions in 
  Section~\ref{sec:conc}.
\section{Third order directional splitting}
\label{sec:approxds}
  We focus here on three-stage exponential Runge--Kutta integrators of
  the form
  \begin{subequations}\label{eq:ERK3}
  \begin{equation}
    \begin{aligned}
    \bb u_{n2}&=\bb u_n+c_2\tau\varphi_1(c_2\tau K)\bb f(t_n,\bb u_n), \\ 
    \bb u_{n3}&=\bb u_n+c_3\tau\varphi_1(c_3\tau K)\bb f(t_n,\bb u_n)
     + \tau a_{32}(\tau K) \bb d_{n2},\\
    \bb u_{n+1}&=\bb u_n+\tau\varphi_1(\tau K)\bb f(t_n,\bb u_n)+
    \tau (b_2(\tau K)\bb d_{n2} + b_3(\tau K)\bb d_{n3}),
    \end{aligned}
    \end{equation}
    where
    \begin{equation}
      \bb d_{ni}=\bb g(t_n+c_i\tau,\bb u_{ni})-\bb g(t_n,\bb u_n), \quad i=2,3,
    \end{equation}
    \end{subequations}
    $c_2$ and $c_3$ are scalars, while $a_{32}(\cdot)$, $b_2(\cdot)$, and
    $b_3(\cdot)$ are (linear combinations of) $\varphi$-functions.
    From Reference~\cite[Section 5.2]{HO05}, we know that we can
    construct integrators of order three
    in which the coefficients $a_{32}(\cdot)$, $b_2(\cdot)$, and $b_3(\cdot)$
    involve just the $\varphi_1$ and $\varphi_2$
    functions. Therefore, we will fully develop our approximation techniques
    using these two functions only. The generalization to higher order
    $\varphi$-functions is straightforward, and is briefly discussed in
    Remark~\ref{rem:phiell}.

    The starting point for the realization of efficient directional split
    exponential integrators is the following formula for the matrix
    exponential of the Kronecker sum
    $K$~\cite{CCEOZ22,CCZ23}
    \begin{equation}\label{eq:expTucker}
      \begin{split}
      \exp(\tau K)\bb v&=\exp(\tau A_{\otimes 1})\cdots\exp(\tau A_{\otimes d})\bb v\\
&=      (\exp(\tau A_d)\otimes\cdots\otimes\exp(\tau A_1))\bb v\\
&=
      \mathrm{vec}(\bb V\times_1\exp(\tau A_1)\times_2\cdots\times_d\exp(\tau A_d)).
\end{split}
      \end{equation}
    Here $\bb V \in \CC^{n_1\times \cdots \times n_d}$ is an order-$d$ tensor such that
    $\mathrm{vec}(\bb V)=\bb v$, {\color{black}being vec the operator which
    stacks by columns the input tensor,} and $\times_\mu$ denotes the $\mu$-mode product,
    i.e., a tensor-matrix product along the direction $\mu$. 
    {\color{black} The concatenation of $\mu$-mode products is referred to as
    Tucker operator.
    Since these are core concepts in the manuscript, we briefly
    describe them in the following. A thorough explanation with full details can
    be found, for instance, in References~\cite{CCZ23,K06,KB09}.
    Given a generic order-$d$ tensor $\bb T\in\CC^{n_1\times\cdots \times n_d}$
    (with elements denoted as $t_{i_1\ldots i_d}$)
    and a matrix $L_\mu \in \CC^{n_\mu\times n_\mu}$ of elements
    $\ell_{i j}^\mu$, the \emph{$\mu$-mode} product of $\bb T$ with $L_\mu$
    (denoted as $\bb T\times_\mu L_\mu$) is the tensor
    $\bb S \in \CC^{n_1\times\cdots \times n_d}$ defined elementwise as
\begin{equation*}
  s_{i_1\ldots i_d} = \sum_{j_\mu=1}^{n_\mu}
                      t_{i_1\ldots i_{\mu-1}j_\mu i_{\mu+1}\ldots i_d}
                      \ell_{i_\mu j_\mu}^\mu.
\end{equation*}
This corresponds to multiplying the matrix
$L_\mu$ onto the \emph{$\mu$-fibers} of the tensor $\bb T$, i.e., 
vectors along direction $\mu$ which are
generalizations to tensors of columns and rows of a matrix.
The concatenation of $\mu$-mode products with the matrices
$L_1,\ldots,L_d$, that is the tensor $\bb S$ with elements
\begin{equation*}
  s_{i_1\ldots i_d} = \sum_{j_d=1}^{n_d}\cdots
                      \sum_{j_1=1}^{n_1}t_{j_1\ldots j_d}
                      \prod_{\mu=1}^d \ell_{i_\mu j_\mu}^\mu,
\end{equation*}
is denoted by $\bb T\times_1 L_1\times_2\cdots\times_d L_d$ and
called \emph{Tucker operator}. In terms of computational cost, a
single $\mu$-mode product requires~$\mathcal{O}(Nn_\mu)$
floating-point operations (with $N=n_1\cdots n_d$),
and it can be implemented by a single (full) matrix-matrix product. In practice, this
can be realized with a single GEMM (GEneral Matrix Multiply) call of level 3 BLAS 
(Basic Linear Algebra Subprograms), 
whose highly optimized
implementations are available essentially for any kind of modern computer 
architecture (we mention, for instance, References~\cite{mkl,XQY12,cublas}).
Consequently, the Tucker operator
has an overall computational cost of $\mathcal{O}(N(n_1+\cdots+n_d))$, 
and it can be realized with $d$ GEMM calls of 
level 3 BLAS.
Finally, the connection between the Kronecker product and the Tucker operator 
is given by the formula (see~\cite[Lemma 2.1]{CCZ23})
\begin{equation}\label{eq:krontomu}
  (L_d \otimes \cdots \otimes L_1)\bb t = 
  \mathrm{vec}(\bb T \times_1 L_1 \times_2 \cdots \times_d L_d), \quad \bb t = \mathrm{vec}(\bb T).
\end{equation}

    Let us now return to the action of the matrix exponential. 
    First of all, notice that in the case $d=2$ formula~\eqref{eq:expTucker}
    simply reduces to 
  \begin{equation}\label{eq:exp2d}
    \exp(\tau(A_{\otimes 1}+A_{\otimes 2}))\bb v=\mathrm{vec}\left(
    \exp(\tau A_1)\bb V\exp(\tau A_2)^{\sf T}\right),
  \end{equation}
  since for order-2 tensors (i.e., matrices) 1- and 2-mode products are just
  the standard matrix-matrix and matrix-matrix transpose multiplications,
  respectively. Hence, after computing the small sized matrix exponential functions,
  formula~\eqref{eq:exp2d} can be efficiently realized with two 
  GEMM calls.
    In general, the
    advantage of the tensor approach in formula~\eqref{eq:expTucker}} is that it 
    allows to compute
    the action $\exp(\tau K)\bb v$ through a single Tucker operator
    without assembling the matrix itself {\color{black}or computing Kronecker products}.
    In fact, we rely
    just on high performance BLAS after the computation of the
    small sized matrix exponentials $\exp(\tau A_\mu)$.
    {\color{black}We remark that the main cost of the procedure lies in the 
    computation of the Tucker operator, while computing the needed matrix
    exponentials is of negligible burden~\cite{CCEOZ22}.}
    Concerning the functions $\varphi_\ell$ with $\ell>0$,
    which are needed for the
    exponential
    integrators considered in this work, we notice that the first equality in
    formula~\eqref{eq:expTucker} is not valid anymore. However, in
    Reference~\cite{CC23} it is shown that
  \begin{equation}\label{eq:secondord}
    \begin{split}
      \varphi_\ell(\tau K)\bb v &=
      \ell!^{d-1}\varphi_\ell(\tau A_{\otimes 1})\cdots\varphi_\ell(\tau A_{\otimes d})
     \bb v + \mathcal{O}(\tau^2)\\
     &=\ell!^{d-1}\left(\varphi_\ell(\tau A_d)\otimes\cdots\otimes\varphi_\ell(\tau A_1)\right)
     \bb v + \mathcal{O}(\tau^2)\\
    &=\mathrm{vec}\left(\ell!^{d-1}\bb V\times_1\varphi_\ell(\tau A_1)\times_2\cdots\times_d
    \varphi_\ell(\tau A_d)\right)+\mathcal{O}(\tau^2),
    \end{split}
  \end{equation}
  of which expressions~\eqref{eq:phisplit} are particular cases. When $d=2$
  the formula simply reduces to
  \begin{equation}\label{eq:phiell2d}
    \varphi_\ell(\tau(A_{\otimes 1}+A_{\otimes 2}))\bb v=\mathrm{vec}\left(
    \varphi_\ell(\tau A_1)(\ell!\bb V)\varphi_\ell(\tau A_2)^{\sf T}\right)
    +\mathcal{O}(\tau^2),
  \end{equation}
  which can be again realized with two GEMM calls after computing the small sized
  $\varphi$-functions.
  Similarly to the action of the matrix exponential, the generic $d$-dimensional
  formulation~\eqref{eq:secondord} requires one Tucker operator and can be implemented with
  $d$ GEMM calls as floating-point operations. Clearly, 
  approximation~\eqref{eq:secondord} cannot be directly employed
  in formulation~\eqref{eq:ERK3}, since it would lead to an order reduction of
  the resulting integrator. In the following subsections, we will then look for
  third order directional split approximations of $\varphi$-functions
  \emph{similar} to formula~\eqref{eq:secondord}, in the sense that their
  realization will require few Tucker operators (which, we emphasize again,
  constitute the major computational cost of computing the approximation).
    \subsection{Two-term two-dimensional splitting}\label{sec:2dds}
    To increase the approximation order of formula~\eqref{eq:phiell2d},
    we look for a combination of the form
  {\color{black}
  \begin{equation}\label{eq:split2d}
  \begin{split}
    \varphi_\ell(\tau(A_{\otimes 1}+A_{\otimes 2}))&=
    \eta_{1,2}\varphi_{\ell_1}(\alpha_{1,2}\tau A_2)\otimes
    \varphi_{\ell_1}(\alpha_{1,1}\tau A_1)\\
    &\phantom{=}+
    \eta_{2,2}\varphi_{\ell_2}(\alpha_{2,2}\tau A_2)\otimes
    \varphi_{\ell_2}(\alpha_{2,1}\tau A_1)+
    \mathcal{O}(\tau^3),
  \end{split}
  \end{equation}
  where $\ell_i>0$ and $\eta_{i,2},\alpha_{i,\mu}\in\CC$
  (with $i=1,2$ and $\mu=1,2$)
  are parameters to be determined. 
  }
    By Taylor expanding, we directly obtain the conditions
{\color{black}
\begin{subequations}\label{eq:phi1phi22d}
\begin{align}
  \frac{\eta_{1,2}}{\ell_1!^2}+\frac{\eta_{2,2}}{\ell_2!^2}&=\frac{1}{\ell!},
  \label{eq:deg0}\\
\frac{\eta_{1,2}\alpha_{1,1}}{\ell_1!(\ell_1+1)!}+
\frac{\eta_{2,2}\alpha_{2,1}}{\ell_2!(\ell_2+1)!}&=\frac{1}{(\ell+1)!},\;
\frac{\eta_{1,2}\alpha_{1,2}}{\ell_1!(\ell_1+1)!}+
\frac{\eta_{2,2}\alpha_{2,2}}{\ell_2!(\ell_2+1)!}=\frac{1}{(\ell+1)!},\label{eq:deg1}\\
\frac{\eta_{1,2}\alpha_{1,1}^2}{\ell_1!(\ell_1+2)!}+
\frac{\eta_{2,2}\alpha_{2,1}^2}{\ell_2!(\ell_2+2)!}&=\frac{1}{(\ell+2)!},\;
\frac{\eta_{1,2}\alpha_{1,2}^2}{\ell_1!(\ell_1+2)!}+
\frac{\eta_{2,2}\alpha_{2,2}^2}{\ell_2!(\ell_2+2)!}=\frac{1}{(\ell+2)!},\label{eq:deg2}\\
\frac{\eta_{1,2}\alpha_{1,1}\alpha_{1,2}}{(\ell_1+1)!^2}+
\frac{\eta_{2,2}\alpha_{2,1}\alpha_{2,2}}{(\ell_2+1)!^2}&=\frac{2}{(\ell+2)!}.\label{eq:deg2ab}
\end{align}
\end{subequations}
}
Then, we have the following result.
\begin{theorem}
  The coefficients in
  Table~\ref{tab:phi1phi22d} are the solutions of system~\eqref{eq:phi1phi22d}
  for $\ell=1,2$ with $\ell_1=1$ and $\ell_2=2$.
\begin{table}[!ht]
\caption{Coefficients for the two-term two-dimensional
  splitting~\eqref{eq:split2d} with $\ell_1=1$ and $\ell_2=2$.}
\label{tab:phi1phi22d}
\centering
\bgroup\def\arraystretch{1.4}
\begin{tabular}{c|cc|cc}
& \multicolumn{2}{c|}{$\ell=1$} & \multicolumn{2}{c}{$\ell=2$}\\
\hline
{\color{black}$\eta_{1,2}$} & $-\frac{5}{4}$ & $\frac{7}{4}\pm\frac{3\sqrt{2}}{2}\rmi$ & $-\frac{4}{3}$ & $\frac{2}{3}\pm\frac{2\sqrt{3}}{3}\rmi$ \\
{\color{black}$\alpha_{1,1}$} & $\pm \frac{4\sqrt{10}}{15}+\frac{4}{3}$ & $\frac{12}{11}\mp\frac{4\sqrt{2}}{11}\rmi$ & $\pm \frac{\sqrt{33}}{8}+\frac{9}{8}$ & $\frac{3}{4}\mp\frac{\sqrt{3}}{4}\rmi$ \\
{\color{black}$\alpha_{1,2}$} & $\mp \frac{4\sqrt{10}}{15}+\frac{4}{3}$ & $\frac{12}{11}\mp\frac{4\sqrt{2}}{11}\rmi$ & $\mp\frac{\sqrt{33}}{8}+\frac{9}{8}$ & $\frac{3}{4}\mp\frac{\sqrt{3}}{4}\rmi$ \\
{\color{black}$\eta_{2,2}$} & $9$ &  $-3\mp 6\sqrt{2}\rmi$ & $\frac{22}{3}$ &  $-\frac{2}{3}\mp\frac{8\sqrt{3}}{3}\rmi$ \\
{\color{black}$\alpha_{2,1}$} & $\pm \frac{2\sqrt{10}}{9}+\frac{16}{9}$ & $\frac{4}{3}\mp \frac{2\sqrt{2}}{3}\rmi$ & $\pm \frac{3\sqrt{33}}{22}+\frac{3}{2}$ & $\frac{6}{7}\mp \frac{3\sqrt{3}}{7}\rmi$\\
{\color{black}$\alpha_{2,2}$} & $\mp\frac{2\sqrt{10}}{9}+\frac{16}{9}$ & $\frac{4}{3}\mp \frac{2\sqrt{2}}{3}\rmi$ & $\mp \frac{3\sqrt{33}}{22}+\frac{3}{2}$ & $\frac{6}{7}\mp \frac{3\sqrt{3}}{7}\rmi$
\end{tabular}
\egroup
\end{table}
\end{theorem}
\begin{proof}
By writing {\color{black}$\eta_{2,2}$} in terms of {\color{black}$\eta_{1,2}$} from equation~\eqref{eq:deg0},
{\color{black}$\alpha_{2,1}$} and {\color{black}$\alpha_{2,2}$} in terms of {\color{black}$\alpha_{1,1}$} and {\color{black}$\alpha_{1,2}$}
from equations~\eqref{eq:deg1}, respectively,
we get from~\eqref{eq:deg2} two equations of degree two
in {\color{black}$\alpha_{1,1}$} and {\color{black}$\alpha_{1,2}$} in terms of {\color{black}$\eta_{1,2}$}
which give four possible pairs.
If we now substitute each pair into equation~\eqref{eq:deg2ab} we get
one equation of degree four
in {\color{black}$\eta_{1,2}$} with no admissible solution, one equation of
degree four in {\color{black}$\eta_{1,2}$} with two complex conjugate solutions,
and two equations of degree one in {\color{black}$\eta_{1,2}$},
each of which has one real solution.
Substituting back gives the desired coefficients.
\end{proof}

Notice that formula~\eqref{eq:split2d} allows for a third order approximation 
of the $\varphi$-functions which,
{\color{black}thanks to equivalence~\eqref{eq:krontomu},}
requires in tensor form the computation of two Tucker operators.
\subsection{Two-term \texorpdfstring{$d$}{d}-dimensional splitting with complex
  coefficients}
In dimension $d$ we consider an approximation of the form
  \begin{equation}\label{eq:splitnd}
  \begin{split}
    \varphi_\ell(\tau(A_{\otimes 1}+\cdots+ A_{\otimes d}))&=
    \eta_{1,d}\varphi_{\ell_1}(\alpha_{1,d}\tau A_d)\otimes\cdots\otimes
    \varphi_{\ell_1}(\alpha_{1,1}\tau A_1)\\
    &\phantom{=}+\eta_{2,d}\varphi_{\ell_2}(\alpha_{2,d}\tau A_d)\otimes\cdots\otimes
    \varphi_{\ell_2}(\alpha_{2,1}\tau A_1)+
    \mathcal{O}(\tau^3),
  \end{split}
  \end{equation}
  which again can be realized in tensor form with two Tucker operators.
  Similarly to the previous case, the coefficients have to satisfy the nonlinear system
\begin{subequations}\label{eq:phi1phi2nd}
\begin{align}
  \frac{\eta_{1,d}}{\ell_1!^d}+\frac{\eta_{2,d}}{\ell_2!^d}&=\frac{1}{\ell!},
  \label{eq:deg0nd}\\
\frac{\eta_{1,d}\alpha_{1,\mu}}{\ell_1!^{d-1}(\ell_1+1)!}+
\frac{\eta_{2,d}\alpha_{2,\mu}}{\ell_2!^{d-1}(\ell_2+1)!}&=\frac{1}{(\ell+1)!},&
\mu&=1,\ldots,d,  \label{eq:deg1and}\\
\frac{\eta_{1,d}\alpha_{1,\mu}^2}{\ell_1!^{d-1}(\ell_1+2)!}+
\frac{\eta_{2,d}\alpha_{2,\mu}^2}{\ell_2!^{d-1}(\ell_2+2)!}&=\frac{1}{(\ell+2)!},&
\mu&=1,\ldots,d,\label{eq:deg2a2nd}\\
\frac{\eta_{1,d}\alpha_{1,\mu}\alpha_{1,\nu}}{\ell_1!^{d-2}(\ell_1+1)!^2}+
\frac{\eta_{2,d}\alpha_{2,\mu}\alpha_{2,\nu}}{\ell_2!^{d-2}(\ell_2+1)!^2}&
=\frac{2}{(\ell+2)!},&\mu,\nu&=1,\ldots,d,\ \mu< \nu.\label{eq:deg2abnd}
\end{align}
\end{subequations}
\begin{theorem}\label{thm:ddimcplx}
  For $\ell=1,2$, $d>2$, $\ell_1=1$, and $\ell_2=2$
  there exists no real solution to system~\eqref{eq:phi1phi2nd}. On the
  other hand, the solutions of system~\eqref{eq:phi1phi2nd} are given by
  the complex coefficients in Table~\ref{tab:phi1phi2nd}.
\begin{table}[!ht]
\caption{Complex coefficients for the two-term $d$-dimensional
    splitting~\eqref{eq:splitnd} with $l_1=1$ and $l_2=2$.}
\label{tab:phi1phi2nd}
\centering
\bgroup\def\arraystretch{1.4}
\begin{tabular}{c|c|c}
& $\ell=1$ & $\ell=2$ \\
\hline
$\eta_{1,d}$ & $\frac{7}{4}\pm\frac{3\sqrt{2}}{2}\rmi$ & $\frac{2}{3}\pm\frac{2\sqrt{3}}{3}\rmi$ \\
$\alpha_{1,\mu}$ & $\frac{12}{11}\mp\frac{4\sqrt{2}}{11}\rmi$ & $\frac{3}{4}\mp\frac{\sqrt{3}}{4}\rmi$ \\
$\eta_{2,d}$ & $2^{d-2}\left(-3\mp 6\sqrt{2}\rmi\right)$ & $2^{d-2}\left(-\frac{2}{3}\mp\frac{8\sqrt{3}}{3}\rmi\right)$ \\
$\alpha_{2,\mu}$ & $\frac{4}{3}\mp \frac{2\sqrt{2}}{3}\rmi$ & $\frac{6}{7}\mp \frac{3\sqrt{3}}{7}\rmi$
\end{tabular}
\egroup
\end{table}
\end{theorem}
\begin{proof}
  Since $\ell_1=1$ and $\ell_2=2$, we have from equation~\eqref{eq:deg0nd} 
  that $\eta_{2,d}=2^d\left(\frac{1}{\ell!}-\eta_{1,d}\right)$.
  Substituting in expressions~\eqref{eq:deg1and}--\eqref{eq:deg2abnd}, we
  obtain a set of equations that does not depend anymore on $d$. Let us then
  consider, among them, the equations for $\mu=\mu_1$ and $\mu=\mu_2$,
  with $\mu_1\neq\mu_2$. Their solutions correspond to the ones obtained for
  the two-dimensional case (see Table~\ref{tab:phi1phi22d} for a summary),
  i.e., {\color{black}$\eta_{1,d}=\eta_{1,2}$,
  $\eta_{2,d}=2^{d-2}\eta_{2,2}$, $\alpha_{1,\mu_1}=\alpha_{1,1}$,
  $\alpha_{1,\mu_2}=\alpha_{1,2}$, $\alpha_{2,\mu_1}=\alpha_{2,1}$,
  and $\alpha_{2,\mu_2}=\alpha_{2,2}$}.
  This has to be valid for each $\mu_1\neq\mu_2$, which
  implies {\color{black}$\alpha_{1,\mu}=\alpha_{1,1}$} and
  {\color{black}$\alpha_{2,\mu}=\alpha_{2,1}$}, for $\mu=1,\ldots,d$, and
  hence excludes the real solutions.
\end{proof}%

{\color{black}Notice that, in the case $d=2$, the coefficients in 
Table~\ref{tab:phi1phi2nd} reduce to the complex coefficients 
in Table~\ref{tab:phi1phi22d}.}
\subsection{Three-term \texorpdfstring{$d$}{d}-dimensional
  splitting with real coefficients}
The approximation in the $d$-dimensional case derived in the previous section
required{\color{black}, in tensor form,} two \emph{complex} Tucker operators to be computed. Here, we present
an alternative formula for $d>2$ which works entirely in \emph{real} arithmetic
but needs three Tucker operators. We then write our
ansatz as
\begin{equation}\label{eq:splitnd3}
  \begin{split}
    \varphi_\ell(\tau(A_{\otimes 1}+\cdots+ A_{\otimes d}))&=
    \eta_{1,d}\varphi_{\ell_1}(\alpha_{1,d}\tau A_d)\otimes\cdots\otimes
    \varphi_{\ell_1}(\alpha_{1,1}\tau A_1)\\
    &\phantom{=}+
    \eta_{2,d}\varphi_{\ell_2}(\alpha_{2,d}\tau A_d)\otimes\cdots\otimes
    \varphi_{\ell_2}(\alpha_{2,1}\tau A_1)\\
    &\phantom{=}+\eta_{3,d}\varphi_{\ell_3}(\alpha_{3,d}\tau A_d)\otimes\cdots\otimes
    \varphi_{\ell_3}(\alpha_{3,1}\tau A_1)+
    \mathcal{O}(\tau^3)
\end{split}
  \end{equation}
  and look for the coefficients which satisfy the nonlinear system
\begin{subequations}\label{eq:phi1phi2nd3t}
\begin{align}
  \frac{\eta_{1,d}}{\ell_1!^d}+\frac{\eta_{2,d}}{\ell_2!^d}
  +\frac{\eta_{3,d}}{\ell_3!^d}&=\frac{1}{\ell!},
  \label{eq:deg0nd3t}\\
\frac{\eta_{1,d}\alpha_{1,\mu}}{\ell_1!^{d-1}(\ell_1+1)!}+
\frac{\eta_{2,d}\alpha_{2,\mu}}{\ell_2!^{d-1}(\ell_2+1)!}+
\frac{\eta_{3,d}\alpha_{3,\mu}}{\ell_3!^{d-1}(\ell_3+1)!}&=\frac{1}{(\ell+1)!},\label{eq:deg1and3t}\\
\frac{\eta_{1,d}\alpha_{1,\mu}^2}{\ell_1!^{d-1}(\ell_1+2)!}+
\frac{\eta_{2,d}\alpha_{2,\mu}^2}{\ell_2!^{d-1}(\ell_2+2)!}+
\frac{\eta_{3,d}\alpha_{3,\mu}^2}{\ell_3!^{d-1}(\ell_3+2)!}&=\frac{1}{(\ell+2)!},\label{eq:deg2a2nd3t}\\
\frac{\eta_{1,d}\alpha_{1,\mu}\alpha_{1,\nu}}{\ell_1!^{d-2}(\ell_1+1)!^2}+
\frac{\eta_{2,d}\alpha_{2,\mu}\alpha_{2,\nu}}{\ell_2!^{d-2}(\ell_2+1)!^2}+
\frac{\eta_{3,d}\alpha_{3,\mu}\alpha_{3,\nu}}{\ell_3!^{d-2}(\ell_3+1)!^2}&
=\frac{2}{(\ell+2)!},\quad \mu< \nu.\label{eq:deg2abnd3t}
\end{align}
In this notation, the indexes $\mu$ and $\nu$ run from $1$ to $d$. Then, we have the following result.
\begin{theorem}
  Let $\ell_1=\ell_3=1$ and $\ell_2=2$. Consider the additional conditions
{\small
  \begin{align}
  \label{eq:addcond}
      \frac{\eta_{1,d}\alpha_{1,\mu}^3}{4!}+\frac{\eta_{2,d}\alpha_{2,\mu}^3}{2!^{d-1}5!}
      +\frac{\eta_{3,d}\alpha_{3,\mu}^3}{4!}&=\frac{1}{(\ell+3)!},\\
      \frac{\eta_{1,d}\alpha_{1,\mu}\alpha_{1,\nu}\alpha_{1,\xi}}{2!^3}+
      \frac{\eta_{2,d}\alpha_{2,\mu}\alpha_{2,\nu}\alpha_{2,\xi}}{2^{d-3}3!^3}
      +\frac{\eta_{3,d}\alpha_{3,\mu}\alpha_{3,\nu}\alpha_{3,\xi}}{2!^3}
      &=\frac{6}{(\ell+3)!}, \quad \mu < \nu < \xi,
  \end{align}
}%
  which correspond to matching the third order terms $A_{\otimes\mu}^3$ and
  $A_{\otimes \mu}A_{\otimes \nu}A_{\otimes \xi}$ in the Taylor expansion.
  Again, in the notation the indexes $\mu$, $\nu$, and $\xi$ run from $1$ to $d$. Then,
  for $\ell=1$ and $\ell=2$
  the coefficients in Table~\ref{tab:phi1phi2nd3} are the only solutions of
  system~\eqref{eq:phi1phi2nd3t}.
\begin{table}[!ht]
  \caption{Real coefficients for the three-term $d$-dimensional
    splitting~\eqref{eq:splitnd3} with $\ell_1=\ell_3=1$ and $\ell_2=2$.}
\label{tab:phi1phi2nd3}
\centering
\bgroup\def\arraystretch{1.4}
  \begin{tabular}{c|c|c}
  & $\ell=1$ & $\ell=2$ \\
\hline
  $\eta_{1,d}$ & $\frac{2243}{1350}\pm \frac{440521}{675\sqrt{2991111}}$& $\frac{19}{27}\pm \frac{151}{27\sqrt{2391}}$ \\
  $\alpha_{1,\mu}$ & $\frac{3(5161 \pm \sqrt{2991111})}{15869}$ & $\frac{3(121 \pm \sqrt{2391})}{490}$ \\
  $\eta_{2,d}$ & $-\frac{12544}{675}\cdot 2^{d-3}$ & $-\frac{196}{27}\cdot 2^{d-3}$ \\
  $\alpha_{2,\mu}$ & $\frac{45}{28}$ & $\frac{9}{7}$ \\
  $\eta_{3,d}$ & $\frac{2243}{1350}\mp \frac{440521}{675\sqrt{2991111}}$ & $\frac{19}{27}\mp \frac{151}{27\sqrt{2391}}$ \\
  $\alpha_{3,\mu}$ & $\frac{3(5161 \mp \sqrt{2991111})}{15869}$ & $\frac{3(121 \mp \sqrt{2391})}{490}$
\end{tabular}
\egroup
\end{table}
\end{theorem}
\end{subequations}
\begin{proof}
  We first consider system~\eqref{eq:phi1phi2nd3t} in the case $d=3$. Then,
  to determine its solutions we use standard arguments of the Gr\"obner basis
  theory~\cite{CLOS15}.
  We assume the lexicographic ordering 
  $\alpha_{1,1}>_{\mathrm{lex}}\alpha_{1,2}>_{\mathrm{lex}}\alpha_{1,3}
  >_{\mathrm{lex}}\alpha_{2,1}>_{\mathrm{lex}}a_{2,2}>_{\mathrm{lex}}\alpha_{2,3}
  >_{\mathrm{lex}}\alpha_{3,1}>_{\mathrm{lex}}a_{3,2}>_{\mathrm{lex}}\alpha_{3,3}
  >_{\mathrm{lex}}\eta_1>_{\mathrm{lex}}\eta_2>_{\mathrm{lex}}\eta_3$,
  and we let $\mu=1,2,3$.
  Then, for $\ell=1$ a Gr\"obner basis of the ideal associated to the system 
  is
  \begin{multline*}
  \Big\{570887639987 - 724578693084\eta_3 + 218051991900 \eta_3^2,
   12544 + 675\eta_2,\\
   -2243 + 675\eta_1 + 675\eta_3,    -45 + 28 a_{2,\mu},
   6486012633 + 13981255498 \alpha_{3,\mu} - 12113999550\eta_3,\\
   -33768359205 + 13981255498 \alpha_{1,\mu} + 12113999550\eta_3\Big\},
  \end{multline*}
  while for $\ell=2$ is
  \begin{multline*}
  \Big\{840350 - 2453166\eta_3 + 1743039 \eta_3^2, 
  196 + 27\eta_2, 
  -38 + 27\eta_1 + 27\eta_3, \\  81474 + 73990 \alpha_{3,\mu} - 193671\eta_3,
  -9 + 7 \alpha_{2,\mu}, 
  -191100 + 73990 \alpha_{1,\mu} + 193671\eta_3\Big\}.
  \end{multline*}
  The desired solutions of system~\eqref{eq:phi1phi2nd3t} are hence equivalently
  given by the zeros of the polynomials in the Gr\"obner basis.
  Simple calculations lead to the coefficients summarized in
  Table~\ref{tab:phi1phi2nd3}.
  Finally, using arguments similar to the ones in the proof of
  Theorem~\ref{thm:ddimcplx}, we get the result for the $d$-dimensional case.
\end{proof}
\begin{remark}\label{rem:phiell}
As mentioned at the beginning of the section, we focused our
attention to the $\varphi_1$ and $\varphi_2$ functions only, since the class of 
integrators that we considered requires at most the latter. Clearly, 
using a similar approach, one could obtain formulas for different
$\ell$ and $\ell_i$.
\end{remark}
{\color{black}
\subsection{Implementation details}\label{sec:impl}
The directional splitting approximations introduced above
(i.e., formulas~\eqref{eq:split2d}, \eqref{eq:splitnd}, and \eqref{eq:splitnd3})
allow for an efficient tensor-oriented approximation of the 
actions of $\varphi$-functions needed in third-order exponential integrators
of the form~\eqref{eq:ERK3}, thanks to equivalence~\eqref{eq:krontomu}.
For the numerical examples presented in Section~\ref{sec:numexp}, we will employ
in particular the time marching scheme with coefficients
$c_2=\frac{1}{3}$, $c_3=\frac{2}{3}$, $a_{32}(\cdot)=\frac{4}{3}\varphi_{2}(c_3\cdot)$,
$b_2(\cdot)=0$, and $b_3(\cdot)=\frac{3}{2}\varphi_2(\cdot)$
(see also tableau (5.8) in Reference~\cite{HO05}). More explicitly, we have
{\small
  \begin{equation}\label{eq:exprk3}
    \begin{aligned}
    \bb u_{n2}&=\bb u_n+\frac{\tau}{3}\varphi_1\left(\frac{\tau}{3} K\right)\bb f(t_n,\bb u_n), \\ 
    \bb u_{n3}&=\bb u_n+\frac{2\tau}{3}\varphi_1\left(\frac{2\tau}{3} K\right)\bb f(t_n,\bb u_n)\\
     &\phantom{=\bb u_n}\;+ \frac{4\tau}{3}\varphi_2\left(\frac{2\tau}{3} K\right)\left(\bb g\left(t_n+\frac{\tau}{3},\bb u_{n2}\right)-\bb g(t_n,\bb u_n)\right),\\
    \bb u_{n+1}&=\bb u_n+\tau\varphi_1(\tau K)\bb f(t_n,\bb u_n)+
    \frac{3\tau}{2} \varphi_2(\tau K)\left(\bb g\left(t_n+\frac{2\tau}{3},\bb u_{n3}\right)-\bb g(t_n,\bb u_n)\right).
    \end{aligned}
    \end{equation}
}%
We label this method \textsc{exprk3} when no directional splitting 
approximations are used.
In this case, in our experiments the needed linear combinations
of actions of $\varphi$-functions
are computed using the very efficient incomplete orthogonalization Krylov-based 
technique described in Reference~\cite{LPR19}.

When the directional splitting approximations are actually employed in 
scheme~\eqref{eq:exprk3}, we label
the resulting methods \textsc{exprk3ds\_real} (if we use real coefficients) and
\textsc{exprk3ds\_cplx} (if we use complex coefficients). 
More in detail, as set of directional splitting coefficients we 
always employ the ones
corresponding to the choice of the symbol $+$ in
$\alpha_{1,1}$ or $\alpha_{1,\mu}$ (see Tables~\ref{tab:phi1phi22d},
\ref{tab:phi1phi2nd}, and \ref{tab:phi1phi2nd3}). 
The use of a different set of coefficients did not provide qualitatively 
different results in the examples.
The pseudocodes of the
directional splitting schemes \textsc{exprk3ds\_real} and \textsc{exprk3ds\_cplx} 
(assuming a constant time step size) are given 
in~\ref{sec:app} (Algorithms~\ref{alg:exprk3dscplx} 
and~\ref{alg:exprk3dsreal}).
Both schemes require
the computation of small sized matrix $\varphi$-functions of the
matrices $A_\mu$ (once and for all before the actual time integration starts).
This is done in practice by employing a rational Pad\'e approach 
with modified scaling and squaring~\cite{SW09} and, as already mentioned previously,
the computational cost of this phase is negligible (see the timings reported
in Section~\ref{sec:numexp}).
Remark also that in the time integration all the relevant
operations are performed in a tensor fashion, without the need to assemble the
matrix $K$ itself or to compute Kronecker products. In fact, the needed 
approximations of actions of $\varphi$-functions
(i.e., application to a vector of formulas~\eqref{eq:split2d}, 
\eqref{eq:splitnd}, and \eqref{eq:splitnd3}) are realized in tensor form
by employing Tucker operators (i.e., exploiting equivalence~\eqref{eq:krontomu}).
Notice that even the action of the matrix $K$ on a vector 
(needed to evaluate $\bb f(t_n,\bb u_n)$)
is realized in tensor form thanks to the equivalence (see~\cite[formula~(9)]{CCZ23})
\begin{equation}\label{eq:kronsumv}
  K\bb t = (A_d \oplus \cdots \oplus A_1) \bb t = 
  \mathrm{vec}\left( \sum_{\mu=1}^d(\bb T \times_\mu A_\mu)\right),
  \quad \mathrm{vec}(\bb T) = \bb t.
\end{equation}
In practice, in our numerical experiments
we compute the Tucker operator and the action of $K$ by exploiting
the high efficiency of level 3 BLAS, as thoroughly explained in 
References~\cite{CCEOZ22,CCZ23}
and briefly summarized after formula~\eqref{eq:expTucker}.
More specifically, when performing experiments in MATLAB environment,
we employ the functions \texttt{tucker} and \texttt{kronsumv} 
from the package KronPACK~\cite{CCZ23}, which directly exploit the multithreaded
MATLAB routines to perform GEMM.
Concerning the experiments in CPU and GPU 
using the C++ and CUDA languages, 
we directly call level 3 BLAS from 
efficient libraries available on the hardware (Intel MKL~\cite{mkl} and 
cuBLAS~\cite{cublas} , respectively)
to realize the needed tensor operations.
Notice that, in this context, we expect consistent speedups in terms of wall-clock time
by employing GPUs instead of CPUs, since these kinds of operations are very well
implemented on the former (see also the discussions in Reference~\cite{CCEOZ22}).
In addition, the evaluation of the nonlinearity 
$\bb g$ and the (pointwise) summation or multiplication operations, 
for instance, are performed by proper 
CUDA kernels, i.e., exploiting a massive parallelism. This is also an area in
which GPUs greatly outperform the corresponding multithreaded version on CPUs.
As a matter of fact, in our GPU implementations of 
Algorithms~\ref{alg:exprk3dscplx} and~\ref{alg:exprk3dsreal}, the only communication
with the CPU is for the solution of the linear systems in the Pad\'e approximation
of the $\varphi$-functions, which was faster if performed on the CPU. This is
expected, since it is a small sized task for which GPUs are not highly optimized.
Apart from that, all the remaining code is executed directly on the GPU.

Finally, remark that to perform a single integration step with
Algorithms~\ref{alg:exprk3dscplx} and~\ref{alg:exprk3dsreal}, we need 
an action of Kronecker sum and 10 and 15 Tucker operators, respectively.
The cost of evaluating Tucker operators, independently of the considered 
hardware, is much lower compared to that of
computing actions of exponentials and/or linear combinations of actions of 
$\varphi$-functions without directional splitting
(as needed by the \textsc{exprk3} integrator~\eqref{eq:exprk3}).
This has already been observed and discussed in full details, for instance, in 
References~\cite{CCEOZ22,CCZ23phi}.
Therefore, we may plausibly expect that the proposed integrators 
\textsc{exprk3ds\_cplx} and \textsc{exprk3ds\_real} will perform well
(compared to \textsc{exprk3}) in terms of simulation wall-clock time,
even if in principle they may introduce a directional splitting error
(see the next section for specific results on the performed numerical experiments).
}
\section{Numerical examples}\label{sec:numexp}
In this section, we present an application of the proposed approximations for 
efficiently solving two popular systems of PDEs, namely the two-component
2D Schnakenberg and 3D FitzHugh--Nagumo models{\color{black}, using
exponential integrators}.
Such models are {\color{black}important} in the context of 
biochemical reactions and electric current flows, since they lead to the formation
of the so-called Turing Patterns~\cite{GLRS19,AMS23,MPV08}.
For both examples, we first perform a semidiscretization in space using
second order uniform centered finite differences, with $n_\mu=n$
discretization points per direction, encapsulating the boundary conditions
{\color{black}(homogeneous Neumann, in fact)} directly
in the relevant matrices. This leads to a system of ODEs in the form
\begin{equation}\label{eq:twocompdisc}
  \begin{bmatrix}
    \bb u'(t) \\ \bb v'(t)
  \end{bmatrix} =
  \begin{bmatrix}
    K_1 & 0\\0 & K_2
  \end{bmatrix}
  \begin{bmatrix}
    \bb u(t) \\ \bb v(t)
  \end{bmatrix} +
  \begin{bmatrix}
    \bb g^1(t,\bb u(t),\bb v(t)) \\ \bb g^2(t,\bb u(t),\bb v(t))
  \end{bmatrix}
  {\color{black}=
  \begin{bmatrix}
    \bb f^1(t,\bb u(t),\bb v(t)) \\ \bb f^2(t,\bb u(t),\bb v(t))
  \end{bmatrix}},
\end{equation}
where {\color{black}$\bb u(t)$ and $\bb v(t)$ represent the two components of
the system of PDEs, while} $K_1$ and $K_2$ are matrices having Kronecker sum structure.
{\color{black}Since the matrix of system~\eqref{eq:twocompdisc} is block
diagonal,} the needed actions of matrix
$\varphi$-functions can be 
performed separately for $K_1$ and $K_2$ using the techniques presented previously
(see also Reference~\cite{CC24}).

For both examples, we perform {\color{black}several experiments which are
briefly described in the following.}
\begin{itemize}
  \item First of all, we test the accuracy of the third order directional split
        integrators \textsc{exprk3ds\_real} and \textsc{exprk3ds\_cplx}, and measure the performances against other methods
        in MATLAB language (using the software MathWorks
        MALTAB\textsuperscript{\textregistered} R2022a). To this aim, we fix the
        number of spatial discretization points, the final simulation time, and 
        we let vary the number of time
        steps. As terms of comparison we consider the second order directional split integrator
        proposed in Reference~\cite{CC24} (that we denote \textsc{etd2rkds}), 
        which corresponds to the popular \textsc{etd2rk} scheme (i.e., tableau (5.3) in
        Reference~\cite{HO05} with $c_2=1$) with second
        order directional splitting of the involved $\varphi$-functions. Moreover,
        we present also the results with the \textsc{etd2rk} and the \textsc{exprk3}
        integrators (without directional splitting) where{\color{black}, as
        mentioned in Section~\ref{sec:impl},} the
        actions of $\varphi$-functions are approximated with the incomplete
        orthogonalization Krylov-based technique~\cite{LPR19} (input
        tolerance and incomplete orthogonalization parameter set to
        $1\mathrm{e}{-6}$ and $2$, respectively). The error is always measured
        in the infinity norm with respect to a reference solution computed with
        the \textsc{exprk3ds\_real} method and a sufficiently large number of
        time steps. The hardware used for
        this experiment is a standard laptop equipped with an Intel Core
        {\color{black}i7-10750H} CPU (6 physical cores) and {\color{black}16GB}
        of RAM.
        {\color{black}On the same hardware, 
        we perform a similar simulation using the C++ language to emphasize
        that the proposed procedures are effective also in this framework.
        More in detail,
        we test the accuracy and the performances of the directional split
        exponential integrators \textsc{etd2rkds}, \textsc{exprk3ds\_real},
        and \textsc{exprk3ds\_cplx} in double precision arithmetic.
        The gcc compiler version is 8.4.0, 
        Intel OneAPI MKL library version 2021.4.0 is used for BLAS, 
        and OpenMP is employed for basic parallelization;}

\item   {\color{black} Then,
        we perform a simulation with the proposed third order integrators
        to show that we are able to retrieve the expected Turing pattern.
        To this aim,
        we fix the number of spatial discretization points,
        we set a large final simulation time, and we let vary the 
        number of time integration steps.
        We measure the performances on the same hardware employed
        previously using
        different software architectures, i.e.,
        MATLAB and C++ 
        with double precision arithmetic.
        We also present results using the CUDA language
        (single precision arithmetic, nvcc compiler and CUDA version 10.1, BLAS
        provided by cuBLAS, and CUDA kernels for massive parallelization).
        In this case, the GPU employed is the mobile one provided with the 
        laptop, i.e., an
        nVIDIA GeForce GTX 1650 Ti card (4GB of dedicated memory). Even if this is
        clearly a consumer-level GPU (and, in particular, not suitable 
        for double precision arithmetic), it can still be effectively employed
        for the relatively small simulations under consideration at this stage;
  }
   \item {\color{black}
   Finally, we present some results on professional CPU and GPU hardware. 
         Once again, we first test the accuracy and the performances of \textsc{etd2rkds}, 
         \textsc{exprk3ds\_real}, and \textsc{exprk3ds\_cplx}
         by fixing the
         number of spatial discretization points, the final simulation time, and 
         by letting vary the number of time steps. Double precision arithmetic
         is employed both for the CPU and the GPU.
         Then, we perform a simulation using the proposed third order integrators 
         for an increasing number of spatial discretization points 
         (i.e., we increase the number of degrees of freedom). In this
         case, we employ double precision arithmetic for the CPU results and
         both double and single precision arithmetic for the GPU ones.
         The CPU hardware is a dual socket Intel Xeon Gold 5118 with
         $2\times 12$ cores, while the GPU is a single nVIDIA V100
         card (equipped with 16 GB of RAM). When calling BLAS on the CPU,
         we use the Intel OneAPI MKL library version 2020.1.0, while on the 
         GPU we employ cuBLAS from CUDA 11.2.
         For the parallelization we again employ
         OpenMP and CUDA kernels for the CPU and the GPU, respectively.}
\end{itemize}
{\color{black} We remark that, in all the examples presented here,
the \textsc{exprk3ds\_cplx} integrator is outperformed by 
the \textsc{exprk3ds\_real} scheme. Nevertheless, we decided to report also
the results of the former for all the experiments. In fact, the 
complex method could still be convenient
in other instances, such as the integration of complex-valued PDEs. We
believe that having an idea of the overall computational cost may then be of 
scientific interest.}

The MATLAB code to reproduce the examples (fully compatible with GNU Octave)
is publicly available in a GitHub 
repository\footnote{Available at \url{https://github.com/cassinif/Expds3}.}.

\subsection{Two-dimensional Schnakenberg model}\label{sec:Schnakenberg_2d}
We consider the following Schnakenberg model in two space dimensions
(see References~\cite{AMS23,DASS20,CC24})
\begin{equation}\label{eq:Schnakenberg_2d}
  \left\{
  \begin{aligned}
    \partial_t u&=
    \delta^u\Delta u+\rho(a^u-u+u^2v),\\
    \partial_t v&=
    \delta^v \Delta v+\rho (a^v-u^2v),
  \end{aligned}\right.
\end{equation}
defined in the spatial domain $\Omega=(0,1)^2$, with homogeneous Neumann
boundary conditions.
The parameters are set to
$\delta^u=1$, $\delta^v=10$, $\rho=1000$, $a^u=0.1$, and $a^v=0.9$
so that the equilibrium $(u_\rme,v_\rme)=(a^u+a^v,a^v/(a^u+a^v)^2)$
is susceptible
of Turing instability. The initial data are
$u_0=u_\rme+10^{-5}\cdot\mathcal{U}(0,1)$ and
$v_0=v_\rme+10^{-5}\cdot\mathcal{U}(0,1)$, where 
$\mathcal{U}(0,1)$ denotes the uniformly distributed random variable
in $(0,1)$. 

We {\color{black}start by performing simulations in MATLAB language} with the spatial domain
discretized by $n=150$ point per direction, i.e., the total number of degrees
of freedom is $N=2 \cdot 150^2$. The final time is set to $T=0.25$.
The number of time steps ranges from 3000 to 6000 for the second order integrators,
while from 1000 to 2500 for the third order ones. The results are collected in
an error decay plot and in a work-precision diagram in Figure~\ref{fig:schnak2d}.
\begin{figure}[htb!]
  \centering
  \input{img/schnak2d_errdiag.tex}\hfill
  \input{img/schnak2d_cpudiag.tex}
  \caption{Results {\color{black}in MATLAB language (standard laptop)} for the simulation of 2D Schnakenberg
  model~\eqref{eq:Schnakenberg_2d} with
  $n=150$ spatial discretization points per direction up to final time $T=0.25$.
  {\color{black}Top} plot: error decay, with reference slope lines of order two (dashed) and 
  three (solid). {\color{black}Bottom} plot: work-precision diagram.}
  \label{fig:schnak2d}
\end{figure}
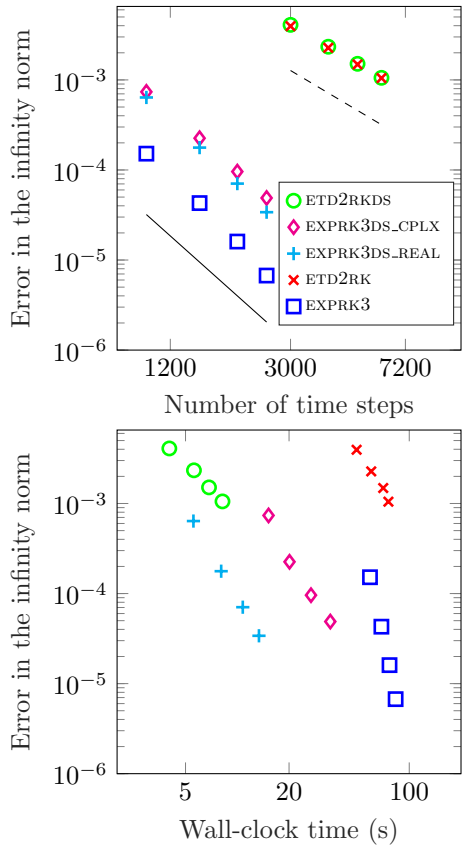
First of all, notice that all the integrators show the expected order of 
convergence. Moreover, we observe that {\color{black}the third order integrators
\textsc{exprk3ds\_cplx} and \textsc{exprk3ds\_real} have larger errors with 
respect to the classical \textsc{exprk3} method, plausibly because of
the introduced directional splitting approximation}.
Nevertheless, the gain in efficiency in terms of wall-clock time is
neat (see {\color{black}right} plot of Figure~\ref{fig:schnak2d}).
In fact, the work-precision diagram is separated into two parts.
On the right hand side we have the most expensive integrators, i.e.,
\textsc{etd2rk} and \textsc{exprk3} implemented with a general-purpose technique 
for computing actions of $\varphi$-functions. On the left hand side we have the
directional split integrators, which always perform better with respect their
original counterparts. Concerning more in detail the third order 
directional split schemes, we observe that \textsc{exprk3ds\_cplx} and 
\textsc{exprk3ds\_real} do not show a considerable difference in terms of 
achieved accuracy. Hence, as expected, the one that employs just real arithmetic is 
cheaper. Overall, for stringent accuracies, the best performant 
integrator is \textsc{exprk3ds\_real}, and in particular it outperforms the
\textsc{etd2rkds} scheme already available in the literature.
{\color{black}We then perform a similar simulation using the C++ language.
We report the results of the directional split integrators
\textsc{etd2rkds}, \textsc{exprk3ds\_real}, and \textsc{exprk3ds\_cplx} in
Figure~\ref{fig:schnak2dCPPmy}.}
\begin{figure}[htb!]
  \centering
  \input{img/schnak2dCPPmy_errdiag.tex}\hfill
  \input{img/schnak2dCPPmy_cpudiag.tex}
  \caption{Results {\color{black}in C++ language (standard laptop)} for the simulation of 2D Schnakenberg
  model~\eqref{eq:Schnakenberg_2d} with
  $n=150$ spatial discretization points per direction up to final time $T=0.25$.
  {\color{black}Top} plot: error decay, with reference slope lines of order two (dashed) and 
  three (solid). {\color{black}Bottom} plot: work-precision diagram.}
  \label{fig:schnak2dCPPmy}
\end{figure}
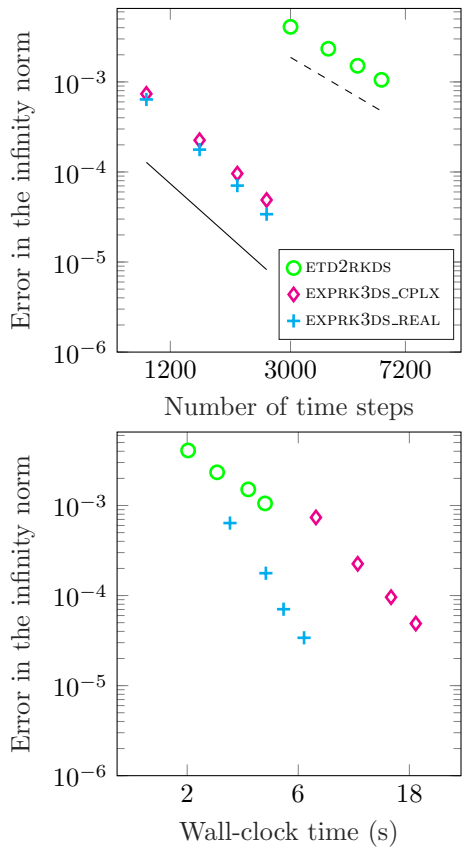
{\color{black}The conclusions are essentially the same
as drawn from the MATLAB experiment, i.e., the integrator that performs best
is \textsc{exprk3ds\_real}.}%

{\color{black} Then,} we compare on different software architectures 
the results of the third order integrators
\textsc{exprk3ds\_real} and {\color{black}\textsc{exprk3ds\_cplx}, i.e.,
the ones that employ the proposed approximations,} to
achieve the expected stationary pattern (a cos-like structure with modes
$(3,5), (5,3)$~\cite{DASS20}).
To this aim, we set the final simulation time to $T=2$ and we let vary the number of 
time steps, while the space discretization is the same as for the previous
experiments (i.e., $n=150$ points per direction).
A representative of the obtained pattern in the $u$ component is shown in 
Figure~\ref{fig:schnak2d_patt}, and the results of the experiment are
collected in Table~\ref{tab:schnak2d_patt_laptop}.
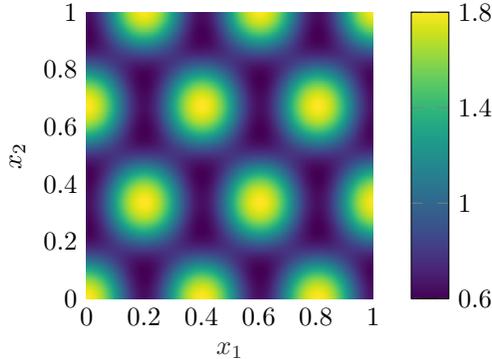
\begin{figure}[htb!]
  \centering
  \input{img/Upatt_schnak_2d_150.tex}
  \caption{Turing pattern ($u$ component)
    for 2D Schnakenberg model~\eqref{eq:Schnakenberg_2d}
    obtained at final time
    $T=2$ with $n=150$ spatial discretization points per direction.}
  \label{fig:schnak2d_patt}
\end{figure}
\begin{table}[htb!]
  \caption{Wall-clock time (in seconds) for the simulation
  of 2D Schnakenberg
  model~\eqref{eq:Schnakenberg_2d} up to $T=2$ with $n=150$
  spatial discretization points per direction (i.e., $N=2\cdot150^2$ degrees
  of freedom).
  {\color{black}The time integrators are \textsc{exprk3ds\_real} (top)
  and \textsc{exprk3ds\_cplx} (bottom)},
  with varying number of time steps, using different software architectures
  {\color{black}(standard laptop)}. 
  In brackets we report the portion of time needed for
  computing the $\varphi$-functions.
  }
  \label{tab:schnak2d_patt_laptop}
  \centering
  \begin{tabular}{c|c|c|c}
  Number of steps & MATLAB \textit{double} & {\color{black}C++} \textit{double} & {\color{black}CUDA} \textit{single} \\
  \hline
  2000 & 10.03 (0.26) & 4.38 (0.21) & 0.93 (0.05)\\
  4000 & 18.87 (0.25) & 8.63 (0.20) & 1.76 (0.06)\\
  6000 & 28.69 (0.27) & 13.09 (0.21) & 2.61 (0.06) \\
  \end{tabular}\\\vspace{5pt}
  \begin{tabular}{c|c|c|c}
  Number of steps & MATLAB \textit{double} & {\color{black}C++} \textit{double} & {\color{black}CUDA} \textit{single} \\
  \hline
  2000 & 23.61 (0.74) & 12.98 (0.61) & 2.41 (0.12)\\
  4000 & 46.11 (0.75) & 26.12 (0.60) & 4.77 (0.11)\\
  6000 & 70.89 (0.77) & 39.20 (0.62) & 6.84 (0.11) \\
  \end{tabular}
\end{table}
Notice that, as expected, the wall-clock time of all the simulations is 
proportional to the number of time steps. In fact, after the computation of 
the small sized $\varphi$-functions, the {\color{black}integrators are}
direct and {\color{black}their} total 
computational burden could be easily predicted by performing a single
time step, similarly to what stated for the \textsc{etd2rkds} scheme
in Reference~\cite{CC24} . Also, remark that the wall-clock time
needed to compute the $\varphi$-functions is negligible compared to that of the
time marching of the methods and, {\color{black}again as expected, 
the \textsc{exprk3ds\_cplx} integrator is computationally more expensive than
\textsc{exprk3ds\_real}}.
We observe that the employment in this context of GPUs is really
effective. In fact, even if we are using a basic consumer-level GPU card,
we can still obtain the
Turing pattern in a short amount of time.

We proceed by presenting some performance results 
doing simulations with professional hardware (see beginning of
Section~\ref{sec:numexp} for the details). 
{\color{black}Similarly to what performed previously, we first test the 
directional split integrators on problem~\eqref{eq:Schnakenberg_2d} with 
the spatial domain
discretized by $n=300$ point per direction, i.e., the total number of degrees
of freedom is $N=2 \cdot 300^2$. The final time is set to $T=0.25$.
The number of time steps ranges from 3000 to 6000 for the second order integrators,
while from 1000 to 2500 for the third order ones. For the computations, we
use double precision arithmetic. The outcome is collected in
the error decay plot and in the work-precision diagram in Figure~\ref{fig:schnak2dser}.}
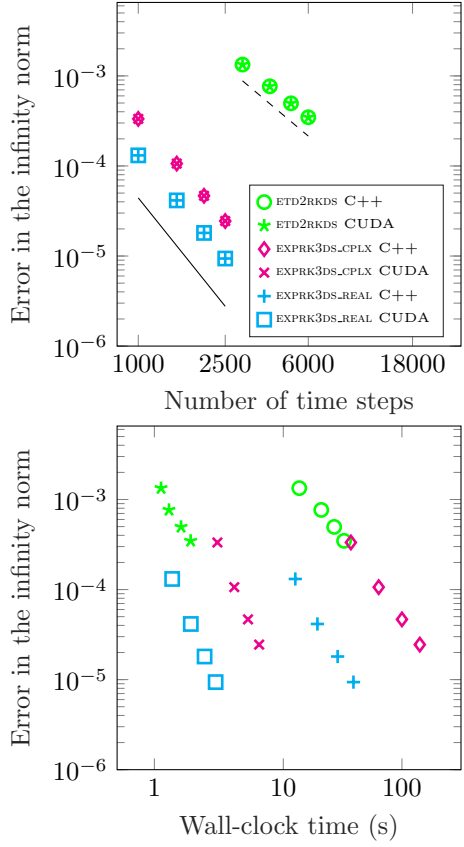
\begin{figure}[htb!]
  \centering
  \input{img/schnak2dCPPser_errdiag.tex}\hfill
  \input{img/schnak2dCPPser_cpudiag.tex}
  \caption{Results {\color{black}in C++ and CUDA languages (professional hardware)} 
  for the simulation of 2D Schnakenberg
  model~\eqref{eq:Schnakenberg_2d} with
  $n=300$ spatial discretization points per direction up to final time $T=0.25$.
  {\color{black}Top} plot: error decay, with reference slope lines of order two (dashed) and 
  three (solid). {\color{black}Bottom} plot: work-precision diagram.}
  \label{fig:schnak2dser}
\end{figure}
{\color{black}As we can see, the results are coherent with what observed 
in the previous experiments using a standard laptop. In particular, all the
integrators have the expected order of convergence, and the specific hardware
employed (CPU or GPU) does not influence the resulting error. On the other hand,
it is clear the effectiveness of the usage of GPUs in this context. In fact, we observe a
speedup of roughly a factor of 10 in favour of GPUs, and the integrator which
performs best \textsc{exprk3ds\_real}.
To conclude, we perform a simulation 
with increasing number of 
degrees of freedom $N$ using the third order 
directional split integrators \textsc{exprk3ds\_real} and \textsc{exprk3ds\_cplx}.
We set the final time to $T=2$, while the number of
time steps is fixed to $6000$.
The results are summarized in Table~\ref{tab:schnak2d_patt_server}.}
\begin{table}[htb!]
  \caption{Wall-clock time (in seconds) for the simulation 
  of 2D Schnakenberg
  model~\eqref{eq:Schnakenberg_2d} up to $T=2$ with increasing number of
  degrees of freedom $N${\color{black}, 6000 time steps,} and different software architectures
  {\color{black} (professional hardware)}. The time 
  {\color{black}integrators are \textsc{exprk3ds\_real} (top)
  and \textsc{exprk3ds\_cplx} (bottom)}.
  In brackets we report the portion of time needed for computing the
  $\varphi$-functions.}
  \label{tab:schnak2d_patt_server}
  \centering
  \begin{tabular}{c|c|c|c}
  Number of d.o.f. $N$& {\color{black}C++} \textit{double} & {\color{black}CUDA} \textit{double} & {\color{black}CUDA} \textit{single} \\
  \hline
  $2\cdot300^2$ & 73.11 (1.85) & 7.32 (0.31)& 4.12 (0.26)\\
  $2\cdot450^2$ & 196.52 (4.72) & 16.39 (0.54) & 7.44 (0.35)\\
  $2\cdot600^2$ & 440.14 (8.73) & 26.15 (0.83) & 13.17 (0.51)\\
  \end{tabular}\\\vspace{5pt}
  \begin{tabular}{c|c|c|c}
  Number of d.o.f. $N$& {\color{black}C++} \textit{double} & {\color{black}CUDA} \textit{double} & {\color{black}CUDA} \textit{single} \\
  \hline
  $2\cdot300^2$ & 192.14 (4.64) & 16.33 (0.52)& 7.90 (0.31)\\
  $2\cdot450^2$ & 607.79 (13.06) & 46.28 (1.01) & 23.91 (0.64)\\
  $2\cdot600^2$ & 1663.05 (23.32) & 89.93 (1.73) & 42.37 (0.98)\\
  \end{tabular}
\end{table}
{\color{black}Again, the integrator employing complex-valued directional splitting
coefficients is more costly than the one with real-valued.}
Comparing CPU and GPU simulations, the scaling in terms of
computational time is still very good (roughly a factor from 10 to 18 for double
precision). This is expected, since the main computational cost in the time
integration comes from the Tucker operators{\color{black}, and hence scales
favourably in GPU.}

\subsection{Three-dimensional
  FitzHugh--Nagumo model}\label{sec:FitzHughNagumo_3d}
We now consider the FitzHugh--Nagumo model in three space dimensions
\begin{equation}\label{eq:FitzHughNagumo_3d}
  \left\{
  \begin{aligned}
    \partial_t u&=
    \delta^u\Delta u+\rho(-u(u^2-1)-v),\\
    \partial_t v&=
    \delta^v \Delta v+\rho a_1^v(u-a_2^vv),
  \end{aligned}\right.
\end{equation}
defined in the spatial domain $\Omega=(0,\pi)^3$ with homogeneous Neumann
boundary conditions. 
The parameters are $\delta^u=1$, $\delta^v=42.1887$,  $\rho=24.649$, $a_1^v=11$, 
and $a_2^v=0.1$. With this choice, the equilibrium $(u_\rme,v_\rme)=(0,0)$ is
susceptible of Turing instability, and we expect to achieve in the long time 
regime a stationary square pattern with modes $(2,2,2)$ (see also
References~\cite{GLRS19,CC24}). The initial conditions are set to 
$u_0=10^{-3}\cdot\mathcal{U}(0,1)$ and $v_0=10^{-3}\cdot\mathcal{U}(0,1)$,
where as in the previous example $\mathcal{U}(0,1)$ denotes the uniform
random variable in the interval $(0,1)$.

For the first experiment, we discretize the spatial domain with a grid of $n=64$ 
points per direction (total number of degrees of freedom $N=2\cdot 64^3$),
and we simulate up to the final time $T=5$ with different integrators
and a number of time steps ranging from $60000$ to $75000$ for the second order method,
while from $14000$ to $20000$ for the third order ones.
The results are graphically depicted in Figure~\ref{fig:fhn3d}.
\begin{figure}[htb!]
  \centering
  \input{img/fhn3d_errdiag.tex}\hfill
  \input{img/fhn3d_cpudiag.tex}
  \caption{Results {\color{black} in MATLAB language (standard laptop)} 
  for the simulation of 3D FitzHugh--Nagumo
  model~\eqref{eq:FitzHughNagumo_3d} with
  $n=64$ spatial discretization points per direction up to final time $T=5$.
  {\color{black}Top} plot: error decay, with reference slope lines of order two (dashed) and 
  three (solid). {\color{black}Bottom} plot: work-precision diagram.}
  \label{fig:fhn3d}
\end{figure}
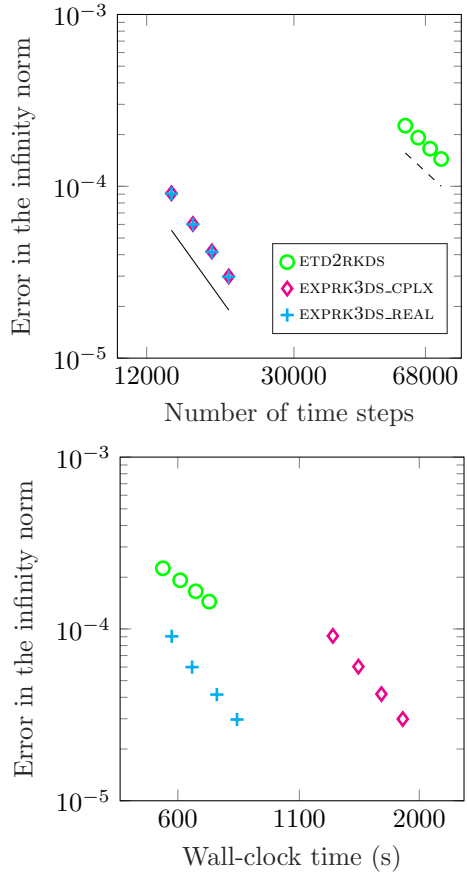
First of all notice that the results with the \textsc{etd2rk} and \textsc{exprk3}
integrators are not reported in the plots, since their simulation wall-clock
time was too large. This is in line with what already observed in the 
two-dimensional example of Section~\ref{sec:Schnakenberg_2d}.
In fact, to obtain comparable accuracies, a simulation with \textsc{etd2rk}
took roughly $6300$ seconds ($60000$ time steps), while \textsc{exprk3}
needed about $6100$ seconds ($14000$ time steps).
Again similarly to the 2D Schnakenberg experiment, the achieved accuracies
of \textsc{exprk3ds\_cplx} and \textsc{exprk3ds\_real} are similar, with an advantage of the
latter in terms of execution time. Overall, for the range of accuracies under
consideration, also in this case the most performant method is \textsc{exprk3ds\_real}.
{\color{black}We repeat the simulations with the directional split exponential
integrators in C++ language (double precision
on a standard laptop). The results, summarized in Figure~\ref{fig:fhn3dCPPmy},
lead essentially to the same conclusions.}
\begin{figure}[htb!]
  \centering
  \input{img/fhn3dCPPmy_errdiag.tex}\hfill
  \input{img/fhn3dCPPmy_cpudiag.tex}
  \caption{Results {\color{black} in C++ language (standard laptop)} 
  for the simulation of 3D FitzHugh--Nagumo
  model~\eqref{eq:FitzHughNagumo_3d} with
  $n=64$ spatial discretization points per direction up to final time $T=5$.
  {\color{black}Top} plot: error decay, with reference slope lines of order two (dashed) and 
  three (solid). {\color{black}Bottom} plot: work-precision diagram.}
  \label{fig:fhn3dCPPmy}
\end{figure}
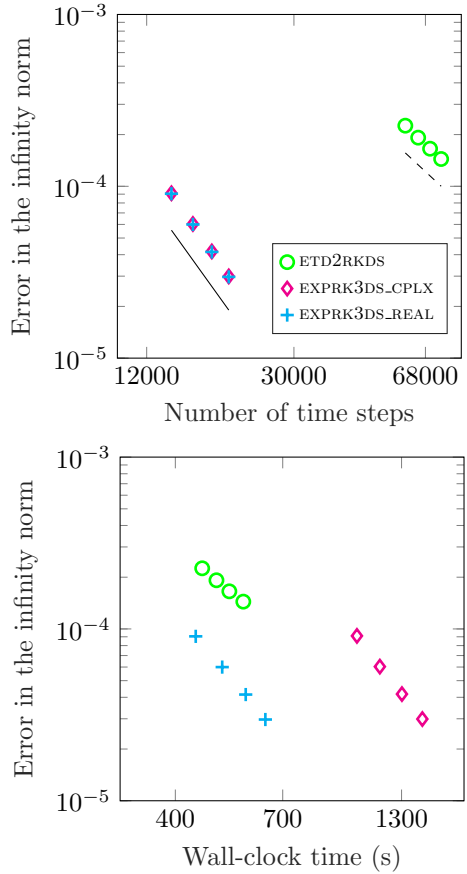

{\color{black} Then, we proceed by comparing the results of 
\textsc{exprk3ds\_real} and \textsc{exprk3ds\_cplx}} to
achieve the expected stationary pattern on different software architectures.
To this aim, we set the final time to $T=150$, we let vary the number of time steps,
while the semidiscretization in space is performed with $n=64$ points per direction.
The obtained pattern in the $u$ component is shown in Figure~\ref{fig:fhn3d_patt}, while the outcome
of the experiment is summarized in Table~\ref{tab:fhn3d_patt_laptop}.
\begin{figure}[htb!]
  \centering
  \input{img/Upatt_fhn_3d_64.tex}
  \caption{Turing pattern ($u$ component)
    for 3D FitzHugh--Nagumo model~\eqref{eq:FitzHughNagumo_3d}
    obtained at final time
    $T=150$ with $n=64$ spatial discretization points per direction.
    The reported slice (left plot) corresponds to $x_3=1.55$
    and the isosurface value (right plot) is $0.08$.}
  \label{fig:fhn3d_patt}
\end{figure}
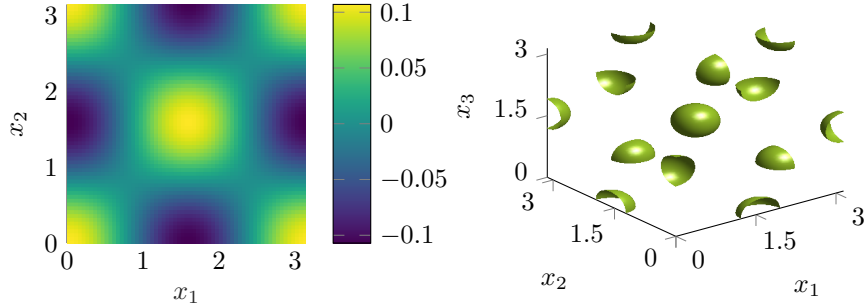
\begin{table}[htb!]
  \caption{Wall-clock time (in seconds) for the simulation 
  of 3D FitzHugh--Nagumo
  model~\eqref{eq:FitzHughNagumo_3d} up to $T=150$ with $n=64$
  spatial discretization points per direction (i.e., $N=2\cdot64^3$ degrees of freedom).
  {\color{black}The time integrators are \textsc{exprk3ds\_real} (top) and
  \textsc{exprk3ds\_cplx} (bottom)},
  with varying number of time steps, using different software architectures
  {\color{black} (standard laptop)}.
  In brackets we report the portion of time needed for computing the
  $\varphi$-functions.}
  \label{tab:fhn3d_patt_laptop}
  \centering
  \begin{tabular}{c|c|c|c}
  Number of steps & MATLAB \textit{double} & {\color{black}C++} \textit{double} & {\color{black}CUDA} \textit{single} \\
  \hline
  10000 & 439.47 (0.09) & 340.36 (0.06) & 26.90 (0.05)\\
  15000 & 607.62 (0.11) & 509.48 (0.07) & 40.92 (0.07)\\
  20000 & 850.41 (0.12) & 677.91 (0.07) & 54.54 (0.07)\\
  \end{tabular}\\\vspace{5pt}
  \begin{tabular}{c|c|c|c}
  Number of steps & MATLAB \textit{double} & {\color{black}C++} \textit{double} & {\color{black}CUDA} \textit{single} \\
  \hline
  10000 & 918.86 (0.13) & 733.68 (0.07) & 44.99 (0.06)\\
  15000 & 1396.22 (0.14) & 1087.42 (0.07) & 65.87 (0.06)\\
  20000 & 1855.47 (0.14) & 1460.29 (0.07) & 90.54 (0.06)\\
  \end{tabular}
\end{table}
Again, as already observed {\color{black}for the 2D Schnakenberg model},
the wall-clock time is proportional to 
the number of time steps for {\color{black} both the integrators} and 
all the architectures under consideration. 
{\color{black}We recall that, in this three dimensional setting, we need to 
compute two \textit{complex} and three \textit{real} Tucker operators 
for each action of $\varphi$-function approximation
in \textsc{exprk3ds\_cplx} and \textsc{exprk3ds\_real}, respectively. 
Even if we need one less Tucker operator
for the former, obviously the computational cost of forming the approximation
is larger, since we employ
complex arithmetic. Overall, we then observe that for the example under
consideration it is more convenient to employ the \textsc{exprk3ds\_real} method.}
Also, notice that the employment of a GPU in this three-dimensional case is
even more effective compared to the two-dimensional scenario.

Finally, we present the performance results by doing simulations with professional
level hardware. To this purpose, 
{\color{black}we first test the 
directional split integrators on problem~\eqref{eq:FitzHughNagumo_3d} with 
the spatial domain
discretized by $n=100$ point per direction, i.e., the total number of degrees
of freedom is $N=2 \cdot 100^3$. The final time is set to $T=5$.
The number of time steps ranges from 60000 to 75000 for the second order integrators,
while from 14000 to 20000 for the third order ones.
We use double precision arithmetic, and the results are collected in Figure~\ref{fig:fhn3dser}.}
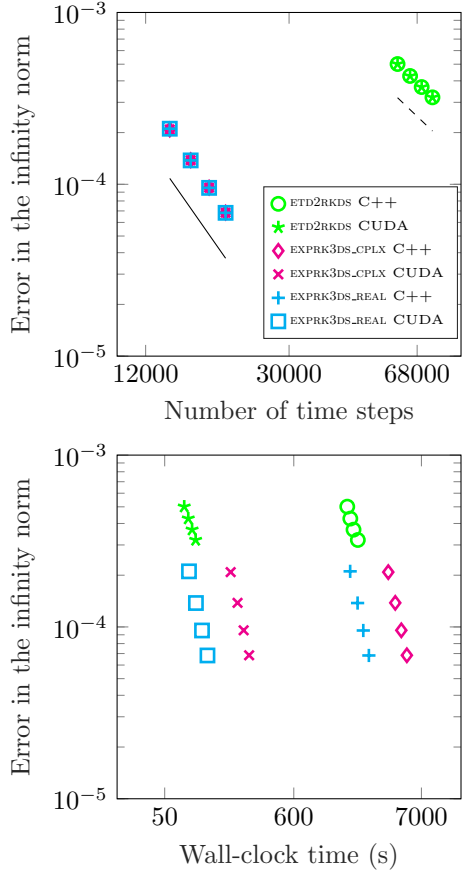
\begin{figure}[htb!]
  \centering
  \input{img/fhn3dCPPser_errdiag.tex}\hfill
  \input{img/fhn3dCPPser_cpudiag.tex}
  \caption{Results {\color{black}in C++ and CUDA languages (professional hardware)}
  for the simulation of 3D FitzHugh--Nagumo
  model~\eqref{eq:FitzHughNagumo_3d} with
  $n=100$ spatial discretization points per direction up to final time $T=5$.
  {\color{black}Top} plot: error decay, with reference slope lines of order two (dashed) and 
  three (solid). {\color{black}Bottom} plot: work-precision diagram.}
  \label{fig:fhn3dser}
\end{figure}
{\color{black}As highlighted also in the previous experiments, we observe 
a clear order of convergence for all the integrators and the architectures under
consideration. Moreover, from the work-precision diagram we conclude that also
in this case it is more convenient to employ the integrator with
real arithmetic than the one with complex coefficients, and overall the integrator
that performs best in reaching stringent accuracies is \textsc{exprk3ds\_real}
implemented in CUDA
(with an average speedup of 22 times passing from the CPU to the GPU).
To conclude, we perform a simulation 
considering the third order integrators \textsc{exprk3ds\_real} and
\textsc{exprk3ds\_cplx} with varying total number of degrees of freedom $N$.}
The final time is set to $T=150$ and the number of
time steps to $10000$.
The results are presented in Table~\ref{tab:fhn3d_patt_server}.
\begin{table}[htb!]
  \caption{Wall-clock time (in seconds) for the simulation
  of 3D
  FitzHugh--Nagumo
  model~\eqref{eq:FitzHughNagumo_3d} up to $T=150$ with increasing number of
  degrees of freedom $N${\color{black}, 10000 time steps,} and different software architectures
  {\color{black} (professional hardware)}. The time
  marching {\color{black}integrators are \textsc{exprk3ds\_real} (top)
  and \textsc{exprk3ds\_cplx} (bottom)}.
  In brackets we report the portion of time needed for computing the
  $\varphi$-functions.}
  \label{tab:fhn3d_patt_server}
  \centering
  \begin{tabular}{c|c|c|c}
  Number of d.o.f. $N$ & {\color{black}C++} \textit{double} & {\color{black}CUDA} \textit{double} & {\color{black}CUDA} \textit{single} \\
  \hline
  $2\cdot100^3$ & 1261.13 (0.29) & 57.31 (0.28) & 32.59 (0.22)\\
  $2\cdot150^3$ & 3686.32 (0.62) & 251.62 (0.40) & 119.60 (0.33)\\
  $2\cdot200^3$ & 9746.87 (1.12) & 685.89 (0.44) & 321.81 (0.35)\\
  \end{tabular}\\\vspace{5pt}
  \begin{tabular}{c|c|c|c}
  Number of d.o.f. $N$ & {\color{black}C++} \textit{double} & {\color{black}CUDA} \textit{double} & {\color{black}CUDA} \textit{single} \\
  \hline
  $2\cdot100^3$ & 2682.31 (0.37) & 127.47 (0.24) & 67.46 (0.22)\\
  $2\cdot150^3$ & 10732.53 (0.90) & 470.61 (0.33) & 241.18 (0.28)\\
  $2\cdot200^3$ & 24640.98 (2.41) & 1469.11 (0.45) & 753.38 (0.32)\\
  \end{tabular}
\end{table}
Also in this final experiment, as expected, the scaling is very good comparing 
CPU and GPU simulations. In particular, the resulting speedup from CPU to 
GPU in double precision is a factor among 14 and 23, hence larger than the 2D 
case. Performing the experiment in single precision arithmetic in GPU allows
to gain an additional factor of 2. {\color{black} The usage of
the scheme \textsc{exprk3ds\_cplx} results in larger 
computational cost also for this example}.

\section{Conclusions}\label{sec:conc}
In this manuscript, we introduced third order directional split approximations
for matrix $\varphi$-functions with underlying $d$-dimensional Kronecker sum 
structure. The derived formulas allow for the efficient construction
and employment of directional split exponential Runge--Kutta integrators of order 
three for the time integration of ODEs systems with Kronecker structure. 
The efficiency of the proposed tensor approach has been tested against 
state-of-the-art techniques on two {\color{black}important} physical models
in the context of Turing patterns for diffusion--reaction systems of
Partial Differential Equations, namely the 2D Schnakenberg and the 3D FitzHugh--Nagumo models.
{\color{black} In particular, it turns out that the investigated 
third order directional split integrator
\textsc{exprk3ds\_real} outperforms the \textsc{etd2rkds} scheme already available
in the literature}.
The numerical experiments also clearly show that the {\color{black}procedures} 
scale very well on modern computer hardware such as Graphic Processing Units.

\section*{Acknowledgments}
{\color{black}
Fabio Cassini received financial support from the Italian 
Ministry of University
and Research (MUR) with the PRIN Project 2022 No. 2022N9BM3N
``Efficient numerical schemes and optimal control methods for time-dependent 
partial differential equations''.
The author is member of the Gruppo Nazionale Calcolo Scientifico-Istituto Nazionale
di Alta Matematica (GNCS-INdAM).}
We would like to thank Marco Caliari for the fruitful discussions
and the valuable suggestions. Also, we are grateful to Lukas Einkemmer and 
the Research Area Scientific Computing of the University of Innsbruck for
the availability to use their professional hardware resources.

\appendix
\section{Pseudocodes}\label{sec:app}
{\color{black}
In the following we list the pseudocodes for the directional split
exponential integrators \textsc{exprk3ds\_real} and \textsc{exprk3ds\_cplx},
which employ the approximations presented in the manuscript.
}%
\begin{algorithm}[htb!]
  \caption{Pseudocode for \textsc{exprk3ds\_real} (if $d=2$) and 
  \textsc{exprk3ds\_cplx}.
  The relevant coefficients for the directional splitting approximations
  are given in Table~\ref{tab:phi1phi22d} (real coefficients) for \textsc{exprk3ds\_real}
  and in Table~\ref{tab:phi1phi2nd} for \textsc{exprk3ds\_cplx}.
  The notations $\alpha_{i,\mu}^{(\ell)}$ 
  and $\eta_{i,d}^{(\ell)}$ mean
  $\alpha_{i,\mu}$ and $\eta_{i,d}$ for the approximation of $\varphi_\ell$.
  The symbols $\mathcal{K}$ and $\mathcal{T}$ denote the action of the 
  Kronecker sum in tensor form~\eqref{eq:kronsumv} and the Tucker
  operator~\eqref{eq:krontomu}, 
  respectively, where the first input 
  is the tensor and the second one is the sequence of matrices.}
  \label{alg:exprk3dscplx}
  \KwIn{Initial datum tensor ($\bb U_0$), 
  ODEs nonlinear function ($\bb g$), 
  matrices which constitute $K$ ($A_1,\ldots, A_d$), final time ($T$), and 
  number of time steps ($m$).}
  \KwOut{Approximated solution $\bb U$ at final time $T$.}
  Compute $\tau=T/m$\;
  \textit{Needed $\varphi$-functions}\\
  \For{$\mu=1,\ldots,d$}{
    Set $A\{\mu\} = A_\mu$\;
    \For{$i=1,2$}{
      \textit{For stage $U_{n2}$}\\
      Compute $P_{i,2}^{(1)}\{\mu\}
      = \varphi_{\ell_i}\left(\frac{\tau}{3}\alpha_{i,\mu}^{(1)}A\{\mu\}\right)$\;
      \For{$\ell=1,2$}{
        \textit{For stage $U_{n3}$}\\
        Compute $P_{i,3}^{(\ell)}\{\mu\}
        = \varphi_{\ell_i}\left(\frac{2\tau}{3}\alpha_{i,\mu}^{(\ell)}A\{\mu\}\right)$\;
        \textit{For final approximation $U_{n+1}$}\\
        Compute $P_{i,f}^{(\ell)}\{\mu\}
        = \varphi_{\ell_i}\left(\tau\alpha_{i,\mu}^{(\ell)}A\{\mu\}\right)$\;
      }
    }
  }
  Set $t=0$\ and $\bb U = \bb U_0$\;
  \textit{Time integration}\\
  \For{$n=0,\ldots,m-1$}{
    Compute $\bb G = \bb g(t,\bb U)$ and $\bb F = \mathcal{K}(\bb U, A) + \bb G$\;
    \textit{Stage $U_{n2}$}\\
    Compute $\bb U_2 = \bb U + \frac{\tau}{3}
    \left(
    \eta_{1,d}^{(1)}\mathcal{T}(\bb F,P_{1,2}^{(1)})+
    \eta_{2,d}^{(1)}\mathcal{T}(\bb F,P_{2,2}^{(1)})
    \right)$\;
    \textit{Stage $U_{n3}$}\\
    Compute $\bb D_2 = \bb g(t+\frac{\tau}{3},\bb U_2) - \bb G$\;
    Compute $\bb U_3 = \bb U + \frac{2\tau}{3}
    \left(
    \eta_{1,d}^{(1)}\mathcal{T}(\bb F,P_{1,3}^{(1)})+
    \eta_{2,d}^{(1)}\mathcal{T}(\bb F,P_{2,3}^{(1)})
    \right)
    + \frac{4\tau}{3}
    \left(
    \eta_{1,d}^{(2)}\mathcal{T}(\bb D_2,P_{1,3}^{(2)})+
    \eta_{2,d}^{(2)}\mathcal{T}(\bb D_2,P_{2,3}^{(2)})
    \right)$\;
    \textit{Final approximation $U_{n+1}$}\\
    Compute $\bb D_3 = \bb g(t+\frac{2\tau}{3},\bb U_3) - \bb G$\;
    Compute $\bb U = \bb U + \tau
    \left(
    \eta_{1,d}^{(1)}\mathcal{T}(\bb F,P_{1,f}^{(1)})+
    \eta_{2,d}^{(1)}\mathcal{T}(\bb F,P_{2,f}^{(1)})
    \right)
    + \frac{3\tau}{2}
    \left(
    \eta_{1,d}^{(2)}\mathcal{T}(\bb D_3,P_{1,f}^{(2)})+
    \eta_{2,d}^{(2)}\mathcal{T}(\bb D_3,P_{2,f}^{(2)})
    \right)$\;
    Set $t = t + \tau$\;
  }
\end{algorithm}
\begin{algorithm}[htb!]
  \caption{Pseudocode for \textsc{exprk3ds\_real} (if $d>2$).
  The relevant coefficients for the directional splitting approximations
  are given in Table~\ref{tab:phi1phi2nd3}.
  The notations $\alpha_{i,\mu}^{(\ell)}$ 
  and $\eta_{i,d}^{(\ell)}$ mean
  $\alpha_{i,\mu}$ and $\eta_{i,d}$ for the approximation of $\varphi_\ell$.
  The symbols $\mathcal{K}$ and $\mathcal{T}$ denote the action of the 
  Kronecker sum in tensor form~\eqref{eq:kronsumv} and the
  Tucker operator~\eqref{eq:krontomu}, 
  respectively, where the first input 
  is the tensor and the second one is the sequence of matrices.}
  \label{alg:exprk3dsreal}
  \KwIn{Initial datum tensor ($\bb U_0$), 
  ODEs nonlinear function ($\bb g$), 
  matrices which constitute $K$ ($A_1,\ldots, A_d$), final time ($T$), and 
  number of time steps ($m$).}
  \KwOut{Approximated solution $\bb U$ at final time $T$.}
  Compute $\tau=T/m$\;
  \textit{Needed $\varphi$-functions}\\
  \For{$\mu=1,\ldots,d$}{
    Set $A\{\mu\} = A_\mu$\;
    \For{$i=1,2,3$}{
      \textit{For stage $U_{n2}$}\\
      Compute $P_{i,2}^{(1)}\{\mu\}
      = \varphi_{\ell_i}\left(\frac{\tau}{3}\alpha_{i,\mu}^{(1)}A\{\mu\}\right)$\;
      \For{$\ell=1,2$}{
        \textit{For stage $U_{n3}$}\\
        Compute $P_{i,3}^{(\ell)}\{\mu\}
        = \varphi_{\ell_i}\left(\frac{2\tau}{3}\alpha_{i,\mu}^{(\ell)}A\{\mu\}\right)$\;
        \textit{For final approximation $U_{n+1}$}\\
        Compute $P_{i,f}^{(\ell)}\{\mu\}
        = \varphi_{\ell_i}\left(\tau\alpha_{i,\mu}^{(\ell)}A\{\mu\}\right)$\;
      }
    }
  }
  Set $t=0$\ and $\bb U = \bb U_0$\;
  \textit{Time integration}\\
  \For{$n=0,\ldots,m-1$}{
    Compute $\bb G = \bb g(t,\bb U)$ and $\bb F = \mathcal{K}(\bb U, A) + \bb G$\;
    \textit{Stage $U_{n2}$}\\
    Compute $\bb U_2 = \bb U + \frac{\tau}{3}
    \left(
    \eta_{1,d}^{(1)}\mathcal{T}(\bb F,P_{1,2}^{(1)})+
    \eta_{2,d}^{(1)}\mathcal{T}(\bb F,P_{2,2}^{(1)})+
    \eta_{3,d}^{(1)}\mathcal{T}(\bb F,P_{3,2}^{(1)})
    \right)$\;
    \textit{Stage $U_{n3}$}\\
    Compute $\bb D_2 = \bb g(t+\frac{\tau}{3},\bb U_2) - \bb G$\;
    Compute $\bb U_3 = \bb U + \frac{2\tau}{3}
    \left(
    \eta_{1,d}^{(1)}\mathcal{T}(\bb F,P_{1,3}^{(1)})+
    \eta_{2,d}^{(1)}\mathcal{T}(\bb F,P_{2,3}^{(1)})+
    \eta_{3,d}^{(1)}\mathcal{T}(\bb F,P_{3,3}^{(1)})
    \right)
    + \frac{4\tau}{3}
    \left(
    \eta_{1,d}^{(2)}\mathcal{T}(\bb D_2,P_{1,3}^{(2)})+
    \eta_{2,d}^{(2)}\mathcal{T}(\bb D_2,P_{2,3}^{(2)})+
    \eta_{3,d}^{(2)}\mathcal{T}(\bb D_2,P_{3,3}^{(2)})
    \right)$\;
    \textit{Final approximation $U_{n+1}$}\\
    Compute $\bb D_3 = \bb g(t+\frac{2\tau}{3},\bb U_3) - \bb G$\;
    Compute $\bb U = \bb U + \tau
    \left(
    \eta_{1,d}^{(1)}\mathcal{T}(\bb F,P_{1,f}^{(1)})+
    \eta_{2,d}^{(1)}\mathcal{T}(\bb F,P_{2,f}^{(1)})+
    \eta_{3,d}^{(1)}\mathcal{T}(\bb F,P_{3,f}^{(1)})
    \right)
    + \frac{3\tau}{2}
    \left(
    \eta_{1,d}^{(2)}\mathcal{T}(\bb D_3,P_{1,f}^{(2)})+
    \eta_{2,d}^{(2)}\mathcal{T}(\bb D_3,P_{2,f}^{(2)})+
    \eta_{3,d}^{(2)}\mathcal{T}(\bb D_3,P_{3,f}^{(2)})
    \right)$\;
    Set $t = t + \tau$\;
  }
\end{algorithm}
\clearpage

\bibliographystyle{elsarticle-num}
\bibliography{main.bib}
\end{document}

%% file: img/schnak2d_errdiag.tex
%
%
%
\begin{tikzpicture}

\begin{axis}[%
width=1.8in,
height=1.8in,
scale only axis,
xmode=log,
xmin=800,
xmax=11000,
xminorticks=true,
xlabel style={font=\color{white!15!black}},
xlabel={Number of time steps},
xtick={1200,3000,7200},
xticklabels={1200,3000,7200},
ymode=log,
ymin=1e-06,
ymax=0.00654314198376849,
yminorticks=true,
ylabel style={font=\color{white!15!black}},
ylabel={Error in the infinity norm},
axis background/.style={fill=white},
legend style={at={(0.98,0.49)}, font=\scriptsize,legend cell align=left, align=left, draw=white!15!black}
]
\addplot [color=green, mark=o, only marks, line width=1pt, mark size = 2.5pt, mark options={solid, green}]
  table[row sep=crcr]{%
3000	4.085037044649316e-03\\
4000	2.335164356336610e-03\\
5000	1.509254859276429e-03\\
6000	1.054866378797149e-03\\
};
\addlegendentry{\textsc{etd2rkds}}

\addplot [color=magenta, mark=diamond, only marks, line width=1pt, mark size = 2.5pt, mark options={solid, magenta}]
  table[row sep=crcr]{%
1000	7.368519354122435e-04\\
1500	2.252011622061814e-04\\
2000	9.593280565580875e-05\\
2500	4.886477837954512e-05\\
};
\addlegendentry{\textsc{exprk3ds\_cplx}}

\addplot [color=cyan, mark=+, only marks, line width=1pt, mark size = 2.5pt, mark options={solid, cyan}]
  table[row sep=crcr]{%
1000	6.383074593443831e-04\\
1500	1.771262449305758e-04\\
2000	7.056477185097450e-05\\
2500	3.397029178243606e-05\\
};
\addlegendentry{\textsc{exprk3ds\_real}}

\addplot [color=red, mark=x, only marks, line width=1pt, mark size = 2.5pt, mark options={solid, red}]
  table[row sep=crcr]{%
3000	3.950993661327285e-03\\
4000	2.272217532763046e-03\\
5000	1.485171585082122e-03\\
6000	1.047674456612400e-03\\
};
\addlegendentry{\textsc{etd2rk}}

\addplot [color=blue, mark=square, only marks, line width=1pt, mark size = 2.5pt, mark options={solid, blue}]
  table[row sep=crcr]{%
1000	1.517866731636416e-04\\
1500	4.287852153680053e-05\\
2000	1.603833297829543e-05\\
2500	6.723653382408810e-06\\
};
\addlegendentry{\textsc{exprk3}}

\addplot [color=black, dashed, forget plot]
  table[row sep=crcr]{%
3000	1.271570991884245e-03\\
6000	3.178927479710610e-04\\
};
\addplot [color=black, forget plot]
  table[row sep=crcr]{%
1000	3.199148154035090e-05\\
2500	2.04745481858246e-06\\
};
\end{axis}

\end{tikzpicture}%

%% file: img/schnak2d_cpudiag.tex
%
%
%
\begin{tikzpicture}

\begin{axis}[%
width=1.8in,
height=1.8in,
scale only axis,
xmode=log,
xmin=2,
xmax=200,
xminorticks=true,
xlabel style={font=\color{white!15!black}},
xlabel={Wall-clock time (s)},
xtick={5,20,100},
xticklabels={5,20,100},
ymode=log,
ymin=1e-06,
ymax=0.00654314198376849,
yminorticks=true,
ylabel style={font=\color{white!15!black}},
ylabel={Error in the infinity norm},
axis background/.style={fill=white},
]

\addplot [color=green, mark=o, only marks, line width=1pt, mark size = 2.5pt, mark options={solid, green}]
  table[row sep=crcr]{%
4.025579	4.085037044649316e-03\\
5.597067	2.335164356336610e-03\\
6.848831	1.509254859276429e-03\\
8.200852	1.054866378797149e-03\\
};

\addplot [color=magenta, mark=diamond, only marks, line width=1pt, mark size = 2.5pt, mark options={solid, magenta}]
  table[row sep=crcr]{%
15.196981	7.368519354122435e-04\\
20.123390	2.252011622061814e-04\\
26.831432	9.593280565580875e-05\\
34.754057	4.886477837954512e-05\\
};

\addplot [color=cyan, mark=+, only marks, line width=1pt, mark size = 2.5pt, mark options={solid, cyan}]
  table[row sep=crcr]{%
5.558783	6.383074593443831e-04\\
8.060007	1.771262449305758e-04\\
10.76247	7.056477185097450e-05\\
13.367402	3.397029178243606e-05\\
};

\addplot [color=red, mark=x, only marks, line width=1pt, mark size = 2.5pt, mark options={solid, red}]
  table[row sep=crcr]{%
49.419626	3.950993661327285e-03\\
60.105015	2.272217532763046e-03\\
70.584855	1.485171585082122e-03\\
75.652296	1.047674456612400e-03\\
};

\addplot [color=blue, mark=square, only marks, line width=1pt, mark size = 2.5pt, mark options={solid, blue}]
  table[row sep=crcr]{%
59.115793	1.517866731636416e-04\\
68.871311	4.287852153680053e-05\\
76.803278	1.603833297829543e-05\\
83.328584	6.723653382408810e-06\\
};

\end{axis}

\end{tikzpicture}%

%% file: img/schnak2dCPPmy_errdiag.tex
%
%
%
\begin{tikzpicture}

\begin{axis}[%
width=1.8in,
height=1.8in,
scale only axis,
xmode=log,
xmin=800,
xmax=11000,
xminorticks=true,
xlabel style={font=\color{white!15!black}},
xlabel={Number of time steps},
xtick={1200,3000,7200},
xticklabels={1200,3000,7200},
ymode=log,
ymin=1e-06,
ymax=0.00654314198376849,
yminorticks=true,
ylabel style={font=\color{white!15!black}},
ylabel={Error in the infinity norm},
axis background/.style={fill=white},
legend style={at={(0.98,0.3)}, font=\scriptsize,legend cell align=left, align=left, draw=white!15!black}
]
\addplot [color=green, mark=o, only marks, line width=1pt, mark size = 2.5pt, mark options={solid, green}]
  table[row sep=crcr]{%
3000	4.085037188610037e-03\\
4000	2.335164260936092e-03\\
5000	1.509254926105934e-03\\
6000	1.054866269470745e-03\\
};
\addlegendentry{\textsc{etd2rkds}}

\addplot [color=magenta, mark=diamond, only marks, line width=1pt, mark size = 2.5pt, mark options={solid, magenta}]
  table[row sep=crcr]{%
1000	7.368521863820337e-04\\
1500	2.252012275530215e-04\\
2000	9.593296351328901e-05\\
2500	4.886469284183583e-05\\
};
\addlegendentry{\textsc{exprk3ds\_cplx}}

\addplot [color=cyan, mark=+, only marks, line width=1pt, mark size = 2.5pt, mark options={solid, cyan}]
  table[row sep=crcr]{%
1000	6.383071828199526e-04\\
1500	1.771260110925291e-04\\
2000	7.056496437208105e-05\\
2500	3.397073242555165e-05\\
};
\addlegendentry{\textsc{exprk3ds\_real}}

\addplot [color=black, dashed, forget plot]
  table[row sep=crcr]{%
3000	0.001873\\
6000	0.0004749\\
};
\addplot [color=black, forget plot]
  table[row sep=crcr]{%
1000	0.000128\\
2500	8.233e-06\\
};
\end{axis}

\end{tikzpicture}%

%% file: img/schnak2dCPPmy_cpudiag.tex
%
%
%
\begin{tikzpicture}

\begin{axis}[%
width=1.8in,
height=1.8in,
scale only axis,
xmode=log,
xmin=1,
xmax=30,
xminorticks=true,
xlabel style={font=\color{white!15!black}},
xlabel={Wall-clock time (s)},
xtick={2,6,18},
xticklabels={2,6,18},
ymode=log,
ymin=1e-06,
ymax=0.00654314198376849,
yminorticks=true,
ylabel style={font=\color{white!15!black}},
ylabel={Error in the infinity norm},
axis background/.style={fill=white},
]

\addplot [color=green, mark=o, only marks, line width=1pt, mark size = 2.5pt, mark options={solid, green}]
  table[row sep=crcr]{%
2.018584077	4.085037188610037e-03\\
2.688800984	2.335164260936092e-03\\
3.659313807	1.509254926105934e-03\\
4.319558642	1.054866269470745e-03\\
};

\addplot [color=magenta, mark=diamond, only marks, line width=1pt, mark size = 2.5pt, mark options={solid, magenta}]
  table[row sep=crcr]{%
7.146398075	7.368521863820337e-04\\
10.802290706	2.252012275530215e-04\\
15.022680035	9.593296351328901e-05\\
19.18399331	4.886469284183583e-05\\
};

\addplot [color=cyan, mark=+, only marks, line width=1pt, mark size = 2.5pt, mark options={solid, cyan}]
  table[row sep=crcr]{%
3.054091073	6.383071828199526e-04\\
4.355233175	1.771260110925291e-04\\
5.188022116	7.056496437208105e-05\\
6.354797426	3.397073242555165e-05\\
};

\end{axis}

\end{tikzpicture}%

%% file: img/Upatt_schnak_2d_150.tex
\begin{tikzpicture}

\begin{axis}[%
width=1.5in,
height=1.5in,
at={(0,0)},
scale only axis,
xmin=0,
xmax=1,
xminorticks=true,
ymin=0,
ymax=1,
ytick style={draw=none},
xtick style={draw=none},
xlabel = {$x_1$},
xlabel style={font=\color{white!15!black}},
ylabel = {$x_2$},
yminorticks=true,
axis equal image,
axis background/.style={fill=white, opacity =0},
axis x line*=bottom,
axis y line*=left,
point meta min=0.6,
point meta max=1.8,
colormap name={viridis},
colorbar,
colorbar style = {ytick={0.6,1,1.4,1.8}}
]
\node at (axis cs: 0.5,0.5) {\includegraphics[width=3.85cm, height=3.85cm]{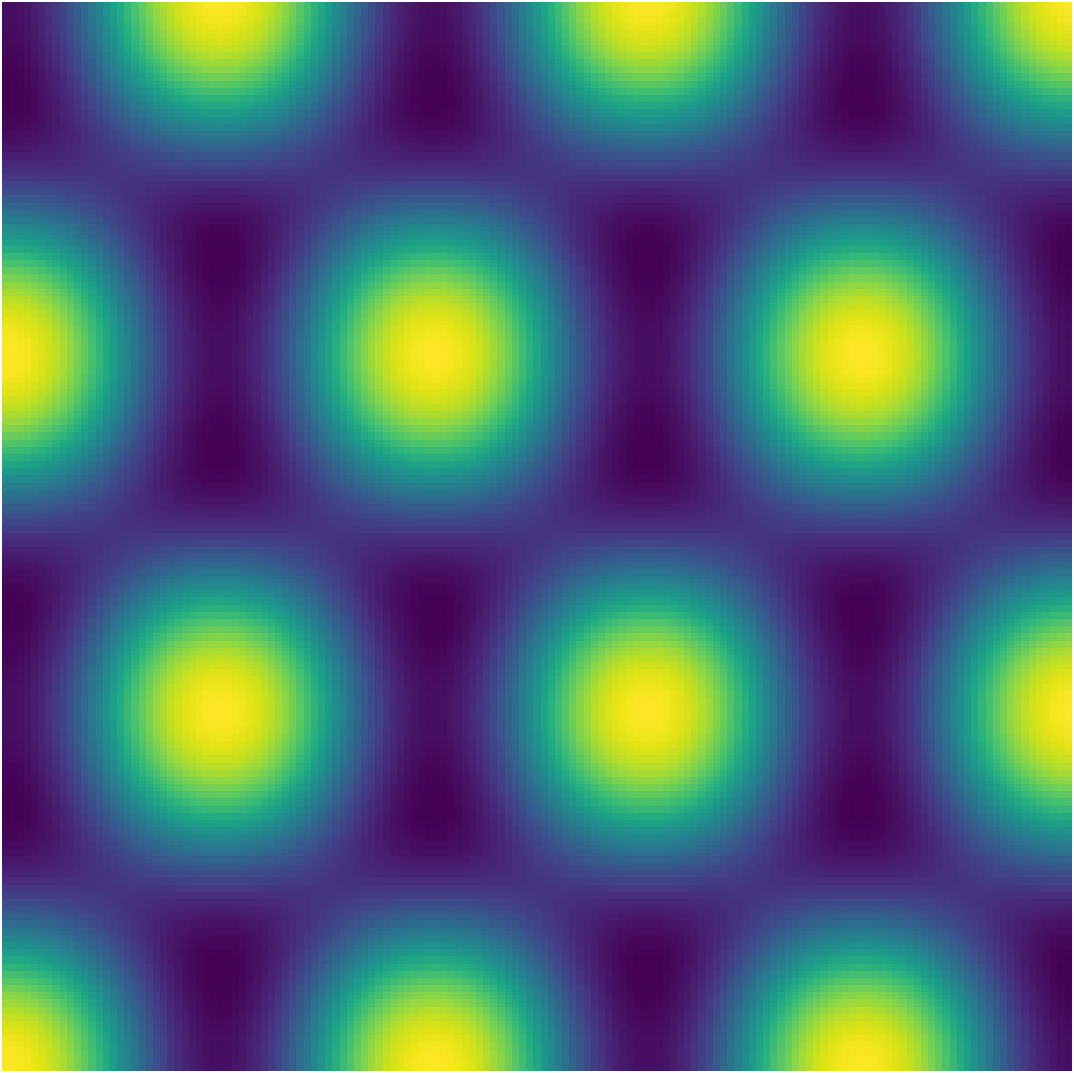}};
\end{axis}

\end{tikzpicture}%

%% file: img/schnak2dCPPser_errdiag.tex
%
%
%
\begin{tikzpicture}

\begin{axis}[%
width=1.8in,
height=1.8in,
scale only axis,
xmode=log,
xmin=800,
xmax=30000,
xminorticks=true,
xlabel style={font=\color{white!15!black}},
xlabel={Number of time steps},
xtick={1000,2500,6000,18000},
xticklabels={1000,2500,6000,18000},
ymode=log,
ymin=1e-06,
ymax=0.00654314198376849,
yminorticks=true,
ylabel style={font=\color{white!15!black}},
ylabel={Error in the infinity norm},
axis background/.style={fill=white},
legend style={at={(0.95,0.47)}, font=\tiny,legend cell align=left, align=left, draw=white!15!black},
legend columns = 1
]
\addplot [color=green, mark=o, only marks, line width=1pt, mark size = 2.5pt, mark options={solid, green}]
  table[row sep=crcr]{%
3000	1.33652105989151e-03\\
4000	7.65949225654664e-04\\
5000	4.96038759644207e-04\\
6000	3.47289405420195e-04\\
};
\addlegendentry{\textsc{etd2rkds} C++}

\addplot [color=green, mark=star, only marks, line width=1pt, mark size = 2.5pt, mark options={solid, green}]
  table[row sep=crcr]{%
3000	1.336522460823274e-03\\
4000	7.659493939373822e-04\\
5000	4.960374417990927e-04\\
6000	3.472884458301285e-04\\
};
\addlegendentry{\textsc{etd2rkds} CUDA}

\addplot [color=magenta, mark=diamond, only marks, line width=1pt, mark size = 2.5pt, mark options={solid, magenta}]
  table[row sep=crcr]{%
1000	3.333224198773761e-04\\
1500	1.063021743437387e-04\\
2000	4.664404335758265e-05\\
2500	2.447341315951764e-05\\
};
\addlegendentry{\textsc{exprk3ds\_cplx} C++}

\addplot [color=magenta, mark=x, only marks, line width=1pt, mark size = 2.5pt, mark options={solid, magenta}]
  table[row sep=crcr]{%
1000	3.333220818713825e-04\\
1500	1.063024434290332e-04\\
2000	4.664404708643686e-05\\
2500	2.447362287140412e-05\\
};
\addlegendentry{\textsc{exprk3ds\_cplx} CUDA}

\addplot [color=cyan, mark=+, only marks, line width=1pt, mark size = 2.5pt, mark options={solid, cyan}]
  table[row sep=crcr]{%
1000	1.311048510188045e-04\\
1500	4.146927621455084e-05\\
2000	1.805537857015891e-05\\
2500	9.382364029157744e-06\\
};
\addlegendentry{\textsc{exprk3ds\_real} C++}

\addplot [color=cyan, mark=square, only marks, line width=1pt, mark size = 2.5pt, mark options={solid, cyan}]
  table[row sep=crcr]{%
1000	1.31105813653203e-04\\
1500	4.146882334351283e-05\\
2000	1.805446440946511e-05\\
2500	9.381644864067653e-06\\
};
\addlegendentry{\textsc{exprk3ds\_real} CUDA}

\addplot [color=black, dashed, forget plot]
  table[row sep=crcr]{%
3000	0.0008818\\
6000	0.00021472\\
};
\addplot [color=black, forget plot]
  table[row sep=crcr]{%
1000	4.4192e-05\\
2500	2.7651e-06\\
};
\end{axis}

\end{tikzpicture}%

%% file: img/schnak2dCPPser_cpudiag.tex
%
%
%
\begin{tikzpicture}

\begin{axis}[%
width=1.8in,
height=1.8in,
scale only axis,
xmode=log,
xmin=0.5,
xmax=300,
xminorticks=true,
xlabel style={font=\color{white!15!black}},
xlabel={Wall-clock time (s)},
xtick={1,11,100},
xticklabels={1,10,100},
ymode=log,
ymin=1e-06,
ymax=0.00654314198376849,
yminorticks=true,
ylabel style={font=\color{white!15!black}},
ylabel={Error in the infinity norm},
axis background/.style={fill=white},
]
\addplot [color=green, mark=o, only marks, line width=1pt, mark size = 2.5pt, mark options={solid, green}]
  table[row sep=crcr]{%
14.82023153	1.33652105989151e-03\\
22.307618726	7.65949225654664e-04\\
28.358623539	4.96038759644207e-04\\
34.070224114	3.47289405420195e-04\\
};

\addplot [color=green, mark=star, only marks, line width=1pt, mark size = 2.5pt, mark options={solid, green}]
  table[row sep=crcr]{%
1.13278625	1.336522460823274e-03\\
1.312132125	7.659493939373822e-04\\
1.643346038	4.960374417990927e-04\\
1.963767974	3.472884458301285e-04\\
};

\addplot [color=magenta, mark=diamond, only marks, line width=1pt, mark size = 2.5pt, mark options={solid, magenta}]
  table[row sep=crcr]{%
38.699848261	3.333224198773761e-04\\
64.923434723	1.063021743437387e-04\\
100.051603871	4.664404335758265e-05\\
139.428945927	2.447341315951764e-05\\
};

\addplot [color=magenta, mark=x, only marks, line width=1pt, mark size = 2.5pt, mark options={solid, magenta}]
  table[row sep=crcr]{%
3.229745274	3.333220818713825e-04\\
4.42176053	1.063024434290332e-04\\
5.707338119	4.664404708643686e-05\\
7.023962272	2.447362287140412e-05\\
};

\addplot [color=cyan, mark=+, only marks, line width=1pt, mark size = 2.5pt, mark options={solid, cyan}]
  table[row sep=crcr]{%
13.759141612	1.311048510188045e-04\\
20.761862037	4.146927621455084e-05\\
30.238310238	1.805537857015891e-05\\
40.58218563	9.382364029157744e-06\\
};

\addplot [color=cyan, mark=square, only marks, line width=1pt, mark size = 2.5pt, mark options={solid, cyan}]
  table[row sep=crcr]{%
1.392616729	1.31105813653203e-04\\
1.970017111	4.146882334351283e-05\\
2.545179511	1.805446440946511e-05\\
3.125902624	9.381644864067653e-06\\
};

\end{axis}

\end{tikzpicture}%

%% file: img/fhn3d_errdiag.tex
%
%
%
\begin{tikzpicture}

\begin{axis}[%
width=1.8in,
height=1.8in,
scale only axis,
xmode=log,
xmin=10000,
xmax=85000,
xminorticks=true,
xlabel style={font=\color{white!15!black}},
xlabel={Number of time steps},
xtick={12000,30000,68000},
xticklabels={12000,30000,68000},
ymode=log,
ymin=1e-05,
ymax=0.001,
yminorticks=true,
ylabel style={font=\color{white!15!black}},
ylabel={Error in the infinity norm},
axis background/.style={fill=white},
legend style={at={(0.45,0.07)}, font=\scriptsize, anchor=south west, legend cell align=left, align=left, draw=white!15!black}
]
\addplot [color=green, mark=o, only marks, line width=1pt, mark size = 2.5pt, mark options={solid, green}]
  table[row sep=crcr]{%
60000	2.253328616585135e-04\\
65000	1.919799740226305e-04\\
70000	1.655145651937596e-04\\
75000	1.441629012672110e-04\\
};
\addlegendentry{\textsc{etd2rkds}}

\addplot [color=magenta, mark=diamond, only marks, line width=1pt, mark size = 2.5pt, mark options={solid, magenta}]
  table[row sep=crcr]{%
14000	9.109767427522430e-05\\
16000	6.034527105270954e-05\\
18000	4.177054921128587e-05\\
20000	2.989602446399220e-05\\
};
\addlegendentry{\textsc{exprk3ds\_cplx}}

\addplot [color=cyan, mark=+, only marks, line width=1pt, mark size = 2.5pt, mark options={solid, cyan}]
  table[row sep=crcr]{%
14000	9.054027169004208e-05\\
16000	5.996967630878248e-05\\
18000	4.150556387181728e-05\\
20000	2.970215271313614e-05\\
};
\addlegendentry{\textsc{exprk3ds\_real}}

\addplot [color=black, dashed, forget plot]
  table[row sep=crcr]{%
60000	1.5643e-04\\
75000	1.0011e-04\\
};
\addplot [color=black, forget plot]
  table[row sep=crcr]{%
14000	5.5510e-05\\
20000	1.9040e-05\\
};
\end{axis}

\end{tikzpicture}%

%% file: img/fhn3d_cpudiag.tex
%
%
%
\begin{tikzpicture}

\begin{axis}[%
width=1.8in,
height=1.8in,
scale only axis,
xmode=log,
xmin=450,
xmax=2500,
xminorticks=true,
xlabel style={font=\color{white!15!black}},
xlabel={Wall-clock time (s)},
xtick={600,1100,2000},
xticklabels={600,1100,2000},
ymode=log,
ymin=1e-05,
ymax=0.001,
yminorticks=true,
ylabel style={font=\color{white!15!black}},
ylabel={Error in the infinity norm},
axis background/.style={fill=white},
]

\addplot [color=green, mark=o, only marks, line width=1pt, mark size = 2.5pt, mark options={solid, green}]
  table[row sep=crcr]{%
5.569551670000000e+02	2.253328616585135e-04\\
6.070850799999999e+02	1.919799740226305e-04\\
6.561293780000000e+02	1.655145651937596e-04\\
7.016587700000000e+02	1.441629012672110e-04\\
};

\addplot [color=magenta, mark=diamond, only marks, line width=1pt, mark size = 2.5pt, mark options={solid, magenta}]
  table[row sep=crcr]{%
1.299962099000000e+03	9.109767427522430e-05\\
1.476715010000000e+03	6.034527105270954e-05\\
1.656050954000000e+03	4.177054921128587e-05\\
1.842145262000000e+03	2.989602446399220e-05\\
};

\addplot [color=cyan, mark=+, only marks, line width=1pt, mark size = 2.5pt, mark options={solid, cyan}]
  table[row sep=crcr]{%
5.822326960000000e+02	9.054027169004208e-05\\
6.443096360000000e+02	5.996967630878248e-05\\
7.290110940000000e+02	4.150556387181728e-05\\
8.061083170000001e+02	2.970215271313614e-05\\
};

\end{axis}

\end{tikzpicture}%

%% file: img/fhn3dCPPmy_errdiag.tex
%
%
%
\begin{tikzpicture}

\begin{axis}[%
width=1.8in,
height=1.8in,
scale only axis,
xmode=log,
xmin=10000,
xmax=85000,
xminorticks=true,
xlabel style={font=\color{white!15!black}},
xlabel={Number of time steps},
xtick={12000,30000,68000},
xticklabels={12000,30000,68000},
ymode=log,
ymin=1e-05,
ymax=0.001,
yminorticks=true,
ylabel style={font=\color{white!15!black}},
ylabel={Error in the infinity norm},
axis background/.style={fill=white},
legend style={at={(0.45,0.07)}, font=\scriptsize, anchor=south west, legend cell align=left, align=left, draw=white!15!black}
]
\addplot [color=green, mark=o, only marks, line width=1pt, mark size = 2.5pt, mark options={solid, green}]
  table[row sep=crcr]{%
60000	0.000225332864830423\\
65000	0.0001919799771860494\\
70000	0.0001655145683752388\\
75000	0.0001441629044752932\\
};
\addlegendentry{\textsc{etd2rkds}}

\addplot [color=magenta, mark=diamond, only marks, line width=1pt, mark size = 2.5pt, mark options={solid, magenta}]
  table[row sep=crcr]{%
14000	9.109767243471983e-05\\
16000	6.034526928561642e-05\\
18000	4.177054762281355e-05\\
20000	2.989602208220342e-05\\
};
\addlegendentry{\textsc{exprk3ds\_cplx}}

\addplot [color=cyan, mark=+, only marks, line width=1pt, mark size = 2.5pt, mark options={solid, cyan}]
  table[row sep=crcr]{%
14000	9.054026954254227e-05\\
16000	5.996967523553991e-05\\
18000	4.150556176911248e-05\\
20000	2.970215058667561e-05\\
};
\addlegendentry{\textsc{exprk3ds\_real}}

\addplot [color=black, dashed, forget plot]
  table[row sep=crcr]{%
60000	1.5643e-04\\
75000	1.0011e-04\\
};
\addplot [color=black, forget plot]
  table[row sep=crcr]{%
14000	5.5510e-05\\
20000	1.9040e-05\\
};
\end{axis}

\end{tikzpicture}%

%% file: img/fhn3dCPPmy_cpudiag.tex
%
%
%
\begin{tikzpicture}

\begin{axis}[%
width=1.8in,
height=1.8in,
scale only axis,
xmode=log,
xmin=300,
xmax=1800,
xminorticks=true,
xlabel style={font=\color{white!15!black}},
xlabel={Wall-clock time (s)},
xtick={400,700,1300},
xticklabels={400,700,1300},
ymode=log,
ymin=1e-05,
ymax=0.001,
yminorticks=true,
ylabel style={font=\color{white!15!black}},
ylabel={Error in the infinity norm},
axis background/.style={fill=white},
]

\addplot [color=green, mark=o, only marks, line width=1pt, mark size = 2.5pt, mark options={solid, green}]
  table[row sep=crcr]{%
460.124832104	0.000225332864830423\\
495.610412575	0.0001919799771860494\\
530.03421668	0.0001655145683752388\\
570.028208723	0.0001441629044752932\\
};

\addplot [color=magenta, mark=diamond, only marks, line width=1pt, mark size = 2.5pt, mark options={solid, magenta}]
  table[row sep=crcr]{%
1030.328105073	9.109767243471983e-05\\
1160.983269914	6.034526928561642e-05\\
1302.577058282	4.177054762281355e-05\\
1448.733082561	2.989602208220342e-05\\
};

\addplot [color=cyan, mark=+, only marks, line width=1pt, mark size = 2.5pt, mark options={solid, cyan}]
  table[row sep=crcr]{%
444.863418024	9.054026954254227e-05\\
510.668947369	5.996967523553991e-05\\
577.2066882189999	4.150556176911248e-05\\
639.452232441	2.970215058667561e-05\\
};

\end{axis}

\end{tikzpicture}%

%% file: img/Upatt_fhn_3d_64.tex
\begin{tikzpicture}

\begin{axis}[%
width=1.25in,
height=1.25in,
at={(0,0)},
scale only axis,
xmin=0,
xmax=pi,
xminorticks=true,
ymin=0,
ymax=pi,
ytick style={draw=none},
xtick style={draw=none},
xlabel = {$x_1$},
xlabel style={font=\color{white!15!black}},
ylabel = {$x_2$},
yminorticks=true,
axis equal image,
axis background/.style={fill=white, opacity =0},
axis x line*=bottom,
axis y line*=left,
point meta min=-0.107,
point meta max=0.107,
colormap name={viridis},
colorbar,
colorbar style = {ytick={-0.1,-0.05,0,0.05,0.1},
        /pgf/number format/fixed,
        /pgf/number format/precision=2,
        scaled y ticks=false}
]
\node at (axis cs: 1.57,1.57) {\includegraphics[width=3.2cm, height=3.2cm]{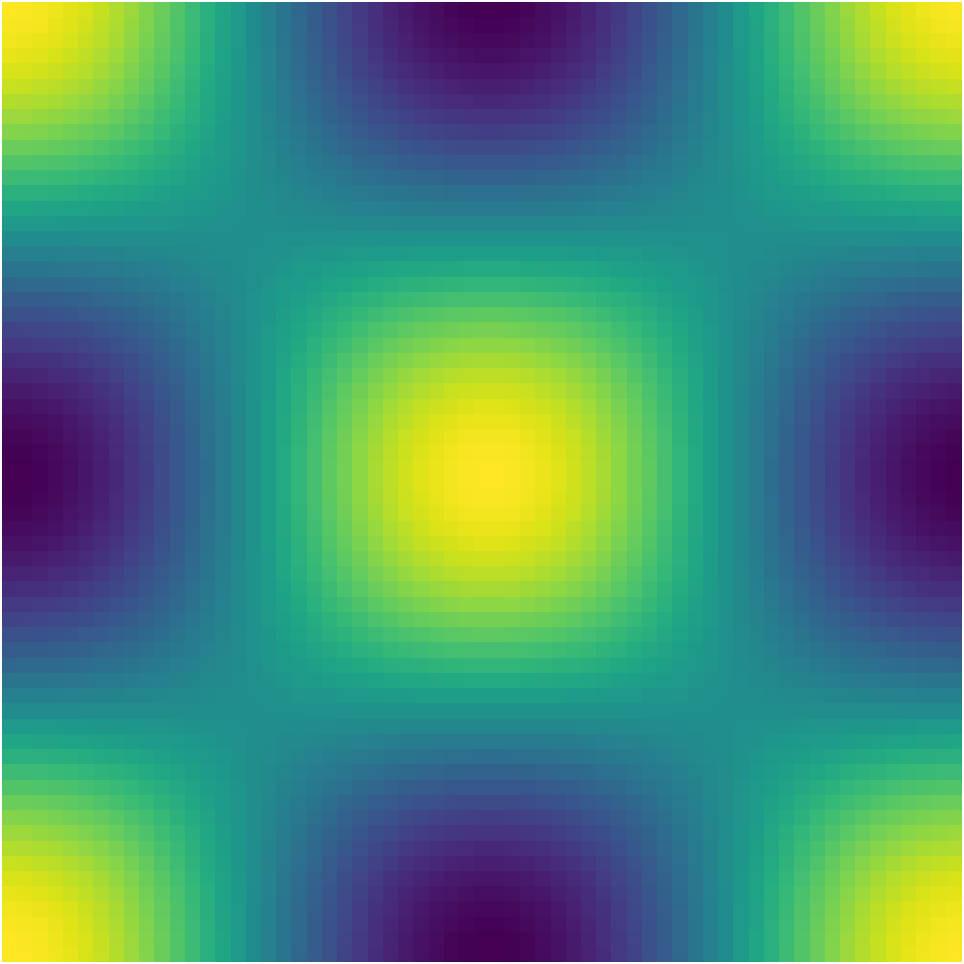}};
\end{axis}

\begin{axis}[%
width=1.5507in,
height=1.2231in,
at={(900,0,-300)},
scale only axis,
xmin=0,
xmax=pi,
tick align=outside,
xtick = {0,1.5,3},
ymin=0,
ymax=pi,
ytick = {0,1.5,3},
zmin=0,
zmax=pi,
ztick = {0,1.5,3},
view={-37.5}{30},
axis background/.style={fill=white},
axis x line*=bottom,
axis y line*=left,
axis z line*=left,
xlabel = {$x_1$},
ylabel = {$x_2$},
zlabel = {$x_3$},
]
\node at (axis cs: 1.57,1.57,1.57) {\includegraphics[scale=0.256]{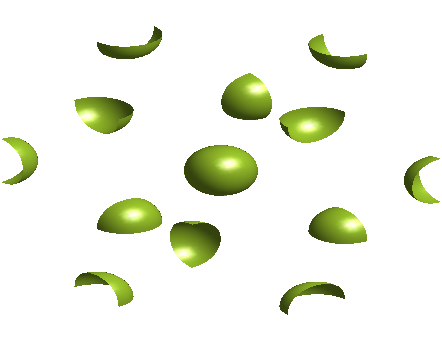}};
\end{axis}

\end{tikzpicture}%

%% file: img/fhn3dCPPser_errdiag.tex
%
%
%
\begin{tikzpicture}

\begin{axis}[%
width=1.8in,
height=1.8in,
scale only axis,
xmode=log,
xmin=10000,
xmax=90000,
xminorticks=true,
xlabel style={font=\color{white!15!black}},
xlabel={Number of time steps},
xtick={12000,30000,68000},
xticklabels={12000,30000,68000},
ymode=log,
ymin=1e-05,
ymax=0.001,
yminorticks=true,
ylabel style={font=\color{white!15!black}},
ylabel={Error in the infinity norm},
axis background/.style={fill=white},
legend style={at={(0.425,0.05)}, font=\tiny, anchor=south west, legend cell align=left, align=left, draw=white!15!black}
]
\addplot [color=green, mark=o, only marks, line width=1pt, mark size = 2.5pt, mark options={solid, green}]
  table[row sep=crcr]{%
60000	0.0005011954702363115\\
65000	0.0004270221398027616\\
70000	0.0003681653987591666\\
75000	0.0003206808619908215\\
};
\addlegendentry{\textsc{etd2rkds} C++}

\addplot [color=green, mark=star, only marks, line width=1pt, mark size = 2.5pt, mark options={solid, green}]
  table[row sep=crcr]{%
60000	0.0005011954702309112\\
65000	0.0004270221397975863\\
70000	0.0003681653987570664\\
75000	0.0003206808620031974\\
};
\addlegendentry{\textsc{etd2rkds} CUDA}

\addplot [color=magenta, mark=diamond, only marks, line width=1pt, mark size = 2.5pt, mark options={solid, magenta}]
  table[row sep=crcr]{%
14000	0.0002084604843785213\\
16000	0.0001380809661048104\\
18000	9.557196352698899e-05\\
20000	6.839706050039597e-05\\
};
\addlegendentry{\textsc{exprk3ds\_cplx} C++}

\addplot [color=magenta, mark=x, only marks, line width=1pt, mark size = 2.5pt, mark options={solid, magenta}]
  table[row sep=crcr]{%
14000	0.000208460485222047\\
16000	0.0001380809665644268\\
18000	9.557196360002263e-05\\
20000	6.839706275996549e-05\\
};
\addlegendentry{\textsc{exprk3ds\_cplx} CUDA}

\addplot [color=cyan, mark=+, only marks, line width=1pt, mark size = 2.5pt, mark options={solid, cyan}]
  table[row sep=crcr]{%
14000	0.0002106528114686799\\
16000	0.0001377542000151196\\
18000	9.534089373529145e-05\\
20000	6.822783495134776e-05\\
};
\addlegendentry{\textsc{exprk3ds\_real} C++}

\addplot [color=cyan, mark=square, only marks, line width=1pt, mark size = 2.5pt, mark options={solid, cyan}]
  table[row sep=crcr]{%
14000	0.0002106528114689049\\
16000	0.0001377542000150446\\
18000	9.534089373206624e-05\\
20000	6.822783495089774e-05\\
};
\addlegendentry{\textsc{exprk3ds\_real} CUDA}

\addplot [color=black, dashed, forget plot]
  table[row sep=crcr]{%
60000	3.1911e-04\\
75000	2.0423e-04\\
};
\addplot [color=black, forget plot]
  table[row sep=crcr]{%
14000	1.0824e-04\\
20000	3.7128e-05\\
};
\end{axis}

\end{tikzpicture}%

%% file: img/fhn3dCPPser_cpudiag.tex
%
%
%
\begin{tikzpicture}

\begin{axis}[%
width=1.8in,
height=1.8in,
scale only axis,
xmode=log,
xmin=20,
xmax=15000,
xminorticks=true,
xlabel style={font=\color{white!15!black}},
xlabel={Wall-clock time (s)},
xtick={50,600,7000},
xticklabels={50,600,7000},
ymode=log,
ymin=1e-05,
ymax=0.001,
yminorticks=true,
ylabel style={font=\color{white!15!black}},
ylabel={Error in the infinity norm},
axis background/.style={fill=white},
]

\addplot [color=green, mark=o, only marks, line width=1pt, mark size = 2.5pt, mark options={solid, green}]
  table[row sep=crcr]{%
1679.959528522	0.0005011954702363115\\
1781.571759107	0.0004270221398027616\\
1895.297879643	0.0003681653987591666\\
2058.60651047	0.0003206808619908215\\
};

\addplot [color=green, mark=star, only marks, line width=1pt, mark size = 2.5pt, mark options={solid, green}]
  table[row sep=crcr]{%
72.810806059	0.0005011954702309112\\
78.692846651	0.0004270221397975863\\
84.71573211899999	0.0003681653987570664\\
90.73100566700001	0.0003206808620031974\\
};

\addplot [color=magenta, mark=diamond, only marks, line width=1pt, mark size = 2.5pt, mark options={solid, magenta}]
  table[row sep=crcr]{%
3699.225876083	0.0002084604843785213\\
4231.275737587	0.0001380809661048104\\
4754.598589253	9.557196352698899e-05\\
5289.477686565	6.839706050039597e-05\\
};

\addplot [color=magenta, mark=x, only marks, line width=1pt, mark size = 2.5pt, mark options={solid, magenta}]
  table[row sep=crcr]{%
178.015848093	0.000208460485222047\\
203.193776211	0.0001380809665644268\\
228.549878098	9.557196360002263e-05\\
253.937193928	6.839706275996549e-05\\
};

\addplot [color=cyan, mark=+, only marks, line width=1pt, mark size = 2.5pt, mark options={solid, cyan}]
  table[row sep=crcr]{%
1775.952824857	0.0002106528114686799\\
2049.277088726	0.0001377542000151196\\
2286.35639239	9.534089373529145e-05\\
2545.965994698	6.822783495134776e-05\\
};

\addplot [color=cyan, mark=square, only marks, line width=1pt, mark size = 2.5pt, mark options={solid, cyan}]
  table[row sep=crcr]{%
79.825662452	0.0002106528114689049\\
91.183593421	0.0001377542000150446\\
102.537909855	9.534089373206624e-05\\
113.903034797	6.822783495089774e-05\\
};

\end{axis}

\end{tikzpicture}%